\documentclass[reqno,11pt]{amsart}
\usepackage[utf8]{inputenc}
\usepackage[margin=1in]{geometry}
\usepackage{amscd,amsfonts,amsmath,amssymb,amsthm}
\usepackage{bbm,dsfont,centernot,mathtools}
\usepackage{color,mathrsfs,stackrel}
\usepackage[colorlinks=true,linkcolor=blue,citecolor=red]{hyperref}
\usepackage{enumitem} 
\usepackage{hyperref} 

\newtheorem{theorem}{Theorem}[section]
\newtheorem{lemma}[theorem]{Lemma}
\newtheorem{proposition}[theorem]{Proposition}
\newtheorem{corollary}[theorem]{Corollary}

\theoremstyle{definition}
\newtheorem{definition}[theorem]{Definition}
\newtheorem{remark}[theorem]{Remark}
\newtheorem{example}[theorem]{Example}

\numberwithin{equation}{section}

\allowdisplaybreaks

\def\supp {\operatorname{supp}}
\def\dist {\operatorname{dist}}
\def\div {\operatorname{div}}
\def\de {\mathrm{d}}
\def\R {\mathbb{R}}
\def\N {\mathbb{N}}

\title[Riesz fractional gradient functionals defined on partitions]{Riesz fractional gradient functionals defined on partitions: Nonlocal-to-local variational limits}

\begin{document}

\author[S.\ Almi]{Stefano Almi}
\address[Stefano Almi]{Department of Mathematics and Applications ``R.~Caccioppoli'', University of Naples Federico II, Via Cintia, Monte S.~Angelo, I-80126 Napoli, Italy}
\email{stefano.almi@unina.it}

\author[M.\ Caponi]{Maicol Caponi}
\address[Maicol Caponi]{Department of Information Engineering, Computer Science and Mathematics, University of L'Aquila, Via Vetoio~1, Coppito, I-67100 L'Aquila, Italy}
\email{maicol.caponi@univaq.it}

\author[M. Friedrich]{Manuel Friedrich} 
\address[Manuel Friedrich]{Department of Mathematics, Johannes Kepler Universit\"at Linz. Altenbergerstrasse 66, 4040 Linz, Austria}
\email{manuel.friedrich@jku.at}

\author[F.\ Solombrino]{Francesco Solombrino}
\address[Francesco Solombrino]{Department of Biological and Environmental Sciences and Technologies, University of Salento, Via Lecce-Monteroni, I-73047 Lecce, Italy}
\email{francesco.solombrino@unisalento.it}

\date{}

\makeatletter
\@namedef{subjclassname@2020}{
\textup{2020} Mathematics Subject Classification}
\makeatother



\begin{abstract}
This paper addresses the asymptotics of functionals with linear growth depending on the Riesz $s$-fractional gradient on piecewise constant functions. We consider a general class of varying energy densities and, as $s\to 1$, we characterize their local limiting functionals in the sense of $\Gamma$-convergence. 
\end{abstract}

\maketitle


\section{Introduction}

Nonlocal functionals represent a powerful alternative to local models, providing a framework to capture long-distance interactions without relying on the existence of gradients. Such a distinctive attribute has proven particularly valuable in multiple applications, where classical local models may fall short. Among others, we mention image processing for advanced edge detection and denoising~\cite{Antil-Bartels, Bessas-Stefani, Gilboa-Osher, Iglesias-Mercier}, data science and machine learning for elucidating data point relationships~\cite{Antiletal, Antiletal2, Bungert-Stinson, Holler-Kunisch}, and mechanical modeling for accurately describing long-range effects in elasticity~\cite{ACFS25, Bellido-MoraCorral-Pedregal, Friedrich-Seitz-Stefanelli, Mengesha-Du, Silhavy20, Silling1, Silling2, Silling3}.

A first approach to nonlocal problems was given in terms of fractional Sobolev spaces~$W^{s, p} (\R^{n})$ for $s \in (0, 1)$ and $p \in [1, +\infty)$, characterized in terms of the Gagliardo seminorm~$[u]_{s, p}$, see~\cite{DiNezza} for an overview. In the case $p=1$, a first notion of fractional perimeter has been introduced in~\cite{Caffarelli} in terms of the Gagliardo seminorm of indicator functions~$[\mathbf{1}_{E}]_{s, 1}$, 
extending the traditional theory of Caccioppoli perimeter~\cite{AFP00}. This led to the development of an extensive theory of, e.g., nonlocal minimal surfaces~\cite{Caffarelli, Cesaronietal, Figalli-Valdinoci}, fractional mean curvature flows~\cite{Cesaroni-DeLuca-Novaga-Ponsiglione, Cesaronietal2, Chambolle-Novaga-Ruffini}, and nonlocal free discontinuity problems~\cite{Carionietal-2025, Cesaroni-Novaga-isoperimeter, Cesaroni-Novaga-clusters, Figalli-Fusco-Maggi-Millot-Morini, Frank-Seiringer}. The asymptotic analysis of the Gagliardo perimeter for $s \to 1$ or $s \to 0$ has been developed in~\cite{Ambrosio-Martinazzi, Fanizza, Ludwig}. 

As noticed in~\cite{ShSp15}, fractional Sobolev spaces do not relate directly to any notion of fractional gradient, and may thus be inadequate for the treatment of space-dependent energy densities. As a workaround, in recent applications to elasticity and image processing~\cite{Antil-Diaz-Jing-Schikorra, Bellido-Cueto-MoraCorral3}, nonlocality has instead been expressed in terms of the Riesz fractional gradient~$\nabla^{s} u$ (see Definition~\ref{def:Riesz-grad} below), studied in~\cite{Horvath, ShSp15, ShSp18}. The Riesz fractional gradient represents a rather natural nonlocal replacement of a gradient as, among others, it possesses good representation formulas in terms of Riesz potentials, as well as useful integration by parts formulas, expressed through a fractional divergence operator ${\rm div}^{s}$ (cf.~Definition~\ref{def:Riesz-grad}). The Riesz fractional gradient has been thoroughly studied in variational problems, with focus on lower semicontinuity, relaxation, and $\Gamma$-convergence~\cite{Bellido-Cueto-MoraCorral4, Bellido-Cueto-MoraCorral2, Cueto-Kreisbeck-Schonberger, KS22} under superlinear growth. We further refer to~\cite{Bellido-Cueto-MoraCorral, CS23-1, Cueto-Kreisbeck-Schonberger2} for the nonlocal-to-local analysis as $s \to 1$.

When coming to the case $p=1$, it has been pointed out in~\cite{CS19, CS23-1} that Gagliardo spaces are not suitable for the definition of a nonlocal version of functions with bounded variation. Such an issue has been addressed in~\cite{BCCS22, CS19, CS23-1}, where the space $BV^{s} (\Omega)$ of functions with $s$-fractional bounded variation in an open set $\Omega \subseteq \R^{n}$ has been introduced by exploiting the fractional divergence operator. This led to the definition of a weak $s$-fractional gradient $D^{s} u \in \mathcal{M}_{b} (\Omega;\R^n)$ for $u \in BV^{s} (\Omega)$. Lower semicontinuity and relaxation issues for functionals with linear growth defined on~$BV^{s} (\R^{n})$ have been discussed in~\cite{Schonberger}. The associated notion of finite fractional $s$-perimeter has been studied extensively in~\cite{CS23-1}. In particular, the authors show that in the nonlocal-to-local limit $s \to 1$ this notion $\Gamma$-converges to the classical Caccioppoli perimeter. 

The goal of the present paper is to initiate a systematic study of $\Gamma$-convergence in fractional spaces of bounded variations in the limit as $s \to 1$, identifying for which scaling of~$s$ a local model is recovered. Inspired by the seminal papers~\cite{AB90-1, AB90-2}, we focus here on the sequence of functionals
\begin{align}
\label{e:intro-Fk}
\mathcal F_k(u) = \int_\Omega \psi_k\left(y,\frac{\de D^{s_k}u}{\de |D^{s_k}u|}(y)\right)\de |D^{s_k}u|(y)
\end{align}
defined on piecewise constant functions $u \colon \R^n \to T$ with finite $s_{k}$-fractional variation, where $T \subset \R$ denotes a finite set. As in~\cite{AB90-1, AB90-2}, our motivation comes from the fact that many problems in physics can be characterized by the formation of partitions which can be described variationally by means of local or nonlocal interfacial energies, e.g., different orientations in liquid crystals or different fluid densities in the theory of Cahn-Hilliard. Compared to the classical setting of~\cite{AB90-1, AB90-2}, the advantage of dealing with fractional gradients is already apparent in the formulation of the problem: fractional gradients are well-defined on low regularity functions and no Geometric Measure Theory tools are required for the definition of~\eqref{e:intro-Fk}. Under suitable constitutive assumptions on the densities~$\psi_{k}$ and the parameter~$s_{k}$, our main result shows that the $\Gamma$-limit of~$\mathcal{F}_{k}$ is a local functional coinciding with the one identified in~\cite{AB90-2} for Caccioppoli partitions, see Theorem~\ref{thm:main}. 

The class of admissible densities $\psi_{k}$ is given by continuous functions in~$\R^{n}\times \R^{n}$, convex and positively 1-homogeneous in the second variable, and uniformly bounded from above and below by a norm in~$\R^{n}$. Moreover, we require a {\em uniform approximation condition}~\eqref{A} which allows to uniformly approximate~$\psi_{k}(x, \xi)$ in terms of finite sums of functions of the form $b_{k}^{i} (x) \varphi_{k}^{i} (\xi)$, $i=1,\ldots,N$, with $b_{k}^{i}\geq 0$ uniformly continuous, and $\varphi_{k}^{i}$ convex and positively 1-homogeneous. 
The precise definition of the admissible class can be found in Section~\ref{sec:main}. Here, we only remark that our framework covers particularly the setting of almost periodic functions considered in~\cite{AB90-2}. The parameter~$s_{k}$ converges to~$1$ under the additional compatibility condition~\eqref{eq:compatibility} involving the radius of uniform continuity~$r_{b_{k}^{i}}$ of the maps~$b_{k}^{i}$, as defined in~\eqref{eq:rb}. As we discuss in Remark~\ref{rem:compatibility}, the convergence rate in~\eqref{eq:compatibility} simplifies in the homogenization setting $\psi_{k} (x, \xi) = \psi (\frac{x}{\varepsilon_{k}}, \xi)$, while similar rates already appeared in the mathematical literature concerning fractional gradients~\cite{ACFS25}. 

 The main idea behind the $\Gamma$-limif inequality is based on the relation between the Riesz fractional gradient $\nabla^{s} u$ and the gradient of the Riesz potential $v= \mathcal{I}^{1-s} u$, namely $\nabla v = \nabla^{s} u$ for $u \in C^{1}_{c} (\R^{n})$, see Proposition~\ref{prop:Riesz} below. In Proposition~\ref{prop:rep} we show that such a property indeed extends, in a suitable sense, to $u \in BV^{s} (\Omega) \cap L^{\infty}(\Omega)$, as we can find $w \in BV(\Omega)$ with~$D w = D^{s} u$. This allows to reformulate~$\mathcal{F}_{k}$ in terms of $BV$-functions. The use of an anisotropic Coarea Formula in Proposition~\ref{prop:app-BVT} enables us to recover piecewise constant $BV$-competitors without substantially changing the energy in the limit $k \to \infty$. At this state, to conclude the lower bound, we can use the $\Gamma$-liminf inequality in~\cite{AB90-2}, or more precisely the one in~\cite{FS20} covering a more general class of densities.

 We emphasize that for the lower bound the uniform approximation condition~\eqref{A} is not needed, but it is essential for the $\Gamma$-limsup inequality. For the latter, a strategy could be to apply Proposition~\ref{prop:rep} the reverse way, building functions in $BV^s$ from the limiting $BV$-function. Yet, the construction of piecewise constant competitors would require a Coarea Formula in $BV^{s} (\Omega)$, a tool which is not available unfortunately, see~\cite[Theorem~3.11 and Corollary~5.6]{CS19}. Hence, we follow another path which consists in showing that recovery sequences for the classical variation~\cite{AB90-2, FS20} are also compatible with the fractional variation. For the approximating densities as in~\eqref{A}, this is achieved by the combination of convexity, Riesz potential, and a duality approach, moving the Riesz operator $\mathcal{I}^{1-s_{k}}$ to the weight functions~$b^{i}_{k}$ appearing in~\eqref{A}. When passing to the limit, the scaling~\eqref{eq:compatibility} between $s_{k}$ and the oscillation parameter~$r_{b_{k}^{i}}$ is necessary to control the error terms quantified, loosely speaking, by $\| \mathcal{I}^{1-s_{k}} b^{i}_{k} - b^{i}_{k}\|_{L^{\infty}}$, see Lemma~\ref{lem:a-app}.

To the best of our knowledge, the present paper is a first step towards a comprehensive theory of $\Gamma$-convergence for fractional energies involving linear growth and nonlocal-to-local effects. An interesting generalization consists in developing the stochastic counterpart of the variational analysis presented in this work. Future research will also be devoted to the extension of our $\Gamma$-convergence result to the case of functionals defined on the whole~$BV^{s}(\Omega)$, thus moving towards the $BV$-theory of~\cite{Cagnettietal}. In this context, let us highlight that our use of Caccioppoli partitions simplifies the formulation of the limit energy drastically, as it only weights the interfaces of the partition and also introduces an a priori $L^{\infty}$-bound on the functional space. The latter is crucially exploited in our arguments, both for compactness (cf.~Theorem~~\ref{thm:main}(1) and Proposition~\ref{prop:comp}) and for cut-off constructions (see Proposition~\ref{prop:app}). Indeed, the $s$-fractional variation of~$u \in BV^{s} (\Omega)$ is compared with the $s$-fractional variation of $\eta u$ over~$\R^{n}$, after the multiplication by a cut-off function~$\eta$. The error produced by this operation is due to a fractional version of Leibniz' rule, and is controlled in terms of the $L^{\infty}$-norm of~$u$.

The paper is organized as follows. Section~\ref{sec:2} is devoted to some preliminaries, in particular the Riesz fractional \texorpdfstring{$s$}{s}-gradient and \texorpdfstring{$s$}{s}-variation. In Section ~\ref{sec:main} we present our setting and formulate the main result. Then, Section~\ref{sec:comp} addresses the proof of compactness, and eventually Section~\ref{sec:Gamma} contains the proof of the $\Gamma$-convergence result. Some auxiliary results are collected in Appendix~\ref{app:a}.


\section{Preliminaries}\label{sec:2}


\subsection{Notation}

For all $x\in\R^n$ and $r>0$, let $B_r(x)\subset\R^n$ denote the open ball of radius $r$ centered at $x$. When $x=0$, we simply write $B_r$, and we set $\mathbb S^{n-1}\coloneq\partial B_1$. For all $x\in\R^n$ and $r>0$, let $Q_r(x)\subset\R^n$ denote an open cube centered at $x$ with sides parallel to the coordinate axes. For all $\nu\in\mathbb S^{n-1}$ we fix an orthogonal matrix $R^\nu\in\R^{n\times n}$ satisfying $R^\nu{\rm e}_n=\nu$, and for all $x\in\R^n$ and $r>0$ we define $Q_r^\nu(x)\coloneq R^\nu Q_r(x)$. When $x=0$, we simply write $Q_r^\nu$. 

Let $\mathcal L^n$ denote the $n$-dimensional Lebesgue measure and $\mathcal H^{n-1}$ the $(n-1)$-dimensional Hausdorff measure. We set $\omega_n\coloneq\mathcal H^{n-1}(\mathbb S^{n-1})$. Given a measurable set $E\subseteq\R^n$, let $\mathbf{1}_E\colon\R^n\to\{0,1\}$ denote its characteristic function. Given two open set $A,B\subset\R^n$, we write $A\subset\subset B$ if there exists a compact set $K$ such that $A\subset K\subset B$. Given two sets $A,B\subseteq\R^n$, we denote their symmetric difference by
$A\triangle B\coloneq (A\setminus B)\cup (B\setminus A)$. 

We adopt standard notation for Lebesgue spaces on measurable subsets $E\subseteq\R^n$. According to the context, we use $\|\cdot\|_{L^p(E)}$ to denote the norm in $L^p(E)$ for all $1\le p\le\infty$. Given an open set $\Omega\subseteq\R^n$ and $k\in\N$, we denote by $\mathcal M(\Omega;\R^k)$ the space of $\R^k$-valued Radon measures on $\Omega$, and by $\mathcal M_b(\Omega;\R^k)$ the subspace of finite Radon measures. Given $\mu\in \mathcal M(\Omega;\R^k)$, we denote by $|\mu|\in\mathcal M(\Omega)$ its total variation, and by $\frac{\de\mu}{\de|\mu|}\in L^1_{\rm loc}(\Omega,|\mu|;\R^k)$ the Radon-Nikodym derivative of $\mu$ with respect to its total variation $|\mu|$.

Given an open set $\Omega\subseteq\R^n$ and a function $u\in L^1_{\rm loc}(\Omega)$, let $Du$ denote the distributional derivative of $u$ in $\Omega$. The space of functions of bounded variation in $\Omega$ is defined as $BV(\Omega)\coloneq\{u\in L^1(\Omega):Du\in\mathcal M_b(\Omega;\R^n)\}$,
and we set $BV_{\rm loc}(\Omega)\coloneq\{u\in L^1_{\rm loc}(\Omega):Du\in \mathcal M(\Omega;\R^n)\}$. 

Let $b\colon\R^n\to\R$ be a uniformly continuous function. We define its \emph{radius of uniform continuity} $r_b\colon(0,\infty)\to(0,1]$ for all $\eta>0$ as
\begin{equation}\label{eq:rb}
r_b(\eta)\coloneq\sup\left\{\delta\in(0,1]:|b(x)-b(y)|\le \eta\text{ for all $x,y\in\R^n$ with $|x-y|\le \delta$}\right\}.
\end{equation}


\subsection{Riesz fractional \texorpdfstring{$s$}{s}-gradient and \texorpdfstring{$s$}{s}-divergence}

We start by recalling the notion of \emph{Riesz $s$-fractional gradient} and \emph{Riesz $s$-fractional divergence}, for $s\in (0,1)$, firstly introduced in~\cite{CS19,ShSp15,Silhavy20}. Let $s\in(0,1)$ be fixed. We define
\begin{align}\label{muu}
\mu_s\coloneq\frac{2^s\Gamma(\frac{n+s+1}{2})}{\pi^\frac{n}{2}\Gamma(\frac{1-s}{2})}\in(0,\infty),
\end{align}
where $\Gamma$ denotes the Gamma function. Notice that
\begin{equation}\label{eq:mus}
\lim_{s\to 1^-}\frac{\mu_s}{1-s}=\frac{n}{\omega_n}=\frac{1}{\mathcal L^n(B_1)}.
\end{equation}

\begin{definition}\label{def:Riesz-grad}
For all $\psi\in C^1_c(\R^n)$ we define the \emph{Riesz $s$-fractional gradient} $\nabla^s\psi\colon\R^n\to\R^n$ as
\begin{equation}\label{eq:nablas}
\nabla^s\psi(x)\coloneq\mu_s\int_{\R^n}\frac{(\psi(y)-\psi(x))(y-x)}{|y-x|^{n+s+1}}\,\de y\quad\text{for all $x\in\R^n$}.
\end{equation}
For all $\Psi\in C^1_c(\R^n;\R^n)$ we define the \emph{Riesz $s$-fractional divergence} $\div^s\Psi\colon\R^n\to\R$ as 
\begin{equation*} 
\div^s\Psi(x)\coloneq\mu_s\int_{\R^n}\frac{(\Psi(y)-\Psi(x))\cdot(y-x)}{|y-x|^{n+s+1}}\,\de y\quad\text{for all $x\in\R^n$}.
\end{equation*} 
\end{definition}

As observed in~\cite[Section~2]{CS19}, the nonlocal operators $\nabla^s\psi$ and $\div^s\Psi$ are well-defined, in the sense that the above integrals converge for all $x\in\R^n$. Moreover, these operators satisfy the following $L^p$-type estimates.

\begin{proposition}\label{prop:gl-est1}
Let $p\in [1,\infty]$. For all $\psi\in C^1_c(\R^n)$ and $\Psi\in C_c^1(\R^n;\R^n)$ we have 
\begin{align}
\|\nabla^s\psi\|_{L^p(\R^n)}&\le \frac{2\omega_n\mu_s}{s(1-s)2^s}\|\psi\|_{L^p(\R^n)}^{1-s}\|\nabla\psi\|_{L^p(\R^n)}^s,\label{eq:est-1}\\
\|\div^s\Psi\|_{L^p(\R^n)}&\le \frac{2\omega_n\mu_s}{s(1-s)2^s}\|\Psi\|_{L^p(\R^n)}^{1-s}\|\nabla\Psi\|_{L^p(\R^n)}^s.\label{eq:est-2}
\end{align}
\end{proposition}

The estimates for $p=\infty$ can be found in~\cite[Lemmas~2.2 and~2.3]{CS23-1}, while analogous estimates for $p\in [1,\infty)$, with slightly different constants, are obtained in~\cite[Propositions~3.2 and~3.3]{CS23-1}. For the reader’s convenience, we give a sketch of the proof in Appendix~\ref{app:a}.

As shown in~\cite[Lemma~2.5]{CS19} and~\cite[Section~6]{Silhavy20}, the nonlocal operators $\nabla^s$ and $\div^s$ satisfy the following integration by parts formula.

\begin{proposition}\label{prop:dual}
For all $\psi\in C^1_c(\R^n)$ and $\Psi\in C^1_c(\R^n;\R^n)$ it holds
\begin{equation}\label{eq:dual}
\int_{\R^n}\nabla^s\psi(y)\cdot\Psi(y)\,\de y=-\int_{\R^n}\psi(y)\div^s\Psi(y)\,\de y.
\end{equation}
\end{proposition}

We next recall the Leibniz-type formulas involving the Riesz $s$-fractional gradient and the Riesz $s$-fractional divergence. To this end, we introduce the following nonlocal operators.

\begin{definition}
For all $\psi,\phi\in C^1_c(\R^n)$ we define the nonlocal operator $\nabla^s_{\rm NL}(\psi,\phi)\colon\R^n\to\R^n$ as
\begin{equation}\label{eq:nablaNL}
\nabla^s_{\rm NL}(\psi,\phi)(x)\coloneq \mu_s\int_{\R^n}\frac{(\phi(y)-\phi(x))(\psi(y)-\psi(x))(y-x)}{|y-x|^{n+s+1}}\,\de y\quad\text{for all $x\in\R^n$}.
\end{equation}
For all $\Psi\in C^1_c(\R^n;\R^n)$ and $\phi\in C^1_c(\R^n)$ we define the nonlocal operator $\div^s_{\rm NL}(\Psi,\phi)\colon\R^n\to\R$ as
\begin{equation}\label{eq:divNL}
\div^s_{\rm NL}(\Psi,\phi)(x)\coloneq \mu_s\int_{\R^n}\frac{(\phi(y)-\phi(x))(\Psi(y)-\Psi(x))\cdot(y-x)}{|y-x|^{n+s+1}}\,\de y\quad\text{for all $x\in\R^n$}.
\end{equation}
\end{definition}

As before, the integrals~\eqref{eq:nablaNL} and~\eqref{eq:divNL} converge for all $x\in\R^n$, and the following $L^p$-type estimates hold.

\begin{proposition}\label{prop:gl-est2}
Let $p\in [1,\infty]$. For all $\psi,\phi\in C^1_c(\R^n)$ and $\Psi\in C_c^1(\R^n;\R^n)$ we have 
\begin{align*}
\|\nabla^s_{\rm NL}(\psi,\phi)\|_{L^p(\R^n)}&\le \frac{4\omega_n\mu_s}{s(1-s)2^s}\|\psi\|_{L^\infty(\R^n)}\|\phi\|^s_{L^p(\R^n)}\|\nabla\phi\|^{1-s}_{L^p(\R^n)},\\
\|\div^s_{\rm NL}(\Psi,\phi)\|_{L^p(\R^n)}&\le \frac{4\omega_n\mu_s}{s(1-s)2^s}\|\Psi\|_{L^\infty(\R^n)}\|\phi\|^s_{L^p(\R^n)}\|\nabla\phi\|^{1-s}_{L^p(\R^n)}.
\end{align*}
\end{proposition}

The proof of Proposition~\ref{prop:gl-est2} is analogous to that of Proposition~\ref{prop:gl-est1} and is therefore omitted. As observed in~\cite[Lemmas~2.6 and~2.7]{CS19}, the following two Leibniz-type formulas hold.

\begin{proposition}\label{prop:nonlocal-leibniz}
For all $\psi,\phi\in C^1_c(\R^n)$ and $\Psi\in C^1_c(\R^n;\R^n)$, we have
\begin{align}
\nabla^s(\psi\phi)(x)&=\psi(x)\nabla^s\phi(x)+\nabla^s\psi(x)\phi(x)+\nabla^s_{\rm NL}(\psi,\phi)(x)& &\text{for all $x\in\R^n$},\label{eq:leibniz}\\
\div^s(\Psi\phi)(x)&=\Psi(x)\cdot\nabla^s\phi(x)+\div^s\Psi(x)\phi(x)+\div^s_{\rm NL}(\Psi,\phi)(x)& &\text{for all $x\in\R^n$}.\label{eq:leibniz2}
\end{align} 
\end{proposition}

We conclude this subsection by recalling the relation between the Riesz $s$-fractional gradient and the classical gradient, as well as between the Riesz $s$-fractional divergence and the classical divergence. To this end, we introduce the notion of the \emph{$\alpha$-Riesz potential} for $\alpha\in(0,n)$. We define
\[
\gamma_\alpha\coloneq\frac{2^\alpha\pi^{\frac{n}{2}} \Gamma\left(\frac{\alpha}{2}\right)}{\Gamma\left(\frac{n-\alpha}{2}\right)}=\frac{n-\alpha }{\mu_{1-\alpha}}\in (0,\infty),
\]
where we note that $\mu_{s}$ in~\eqref{muu} can also be defined for $s \in (1-n,0]$. In view of~\eqref{eq:mus}, we have
\begin{equation}\label{eq:gammaa}
\lim_{\alpha\to 0^+}\alpha\gamma_\alpha=\omega_n.
\end{equation}

\begin{definition}
Let $\mu\in\mathcal M(\R^n)$ be a Radon measure satisfying
\begin{equation}\label{eq:Ia}
\int_{\R^n}\frac{\de|\mu|(y)}{(1+|y|)^{n-\alpha}}<\infty.
\end{equation}
For $\alpha \in (0,n)$, we define the \emph{$\alpha$-Riesz potential} $\mathcal I^\alpha \mu\colon\R^n\to\R$ as
\[
\mathcal I^\alpha \mu(x)\coloneq \frac{1}{\gamma_\alpha}\int_{\R^n}\frac{\de \mu(y)}{|y-x|^{n-\alpha}}\quad\text{for a.e.\ $x\in\R^n$}.
\]
When $\mu=f\,\mathcal L^n$, we simply write $\mathcal I^\alpha f$ instead of $\mathcal I^\alpha \mu$.
\end{definition}

The operator $\mathcal I^\alpha$ is well-defined in view of the following result.

\begin{proposition}[{\cite[Theorem~1.1, Chapter~2]{Mizuta96}}]\label{prop:mizuta}
Let $\mu\in\mathcal M(\R^n)$ be a Radon measure satisfying~\eqref{eq:Ia}. Then, $\mathcal I^\alpha\mu(x)$ is well-defined for a.e.\ $x\in\R^n$ and $\mathcal I^\alpha\mu\in L^1_{\rm loc}(\R^n)$.
\end{proposition}



As shown in~\cite[Theorem~1.2]{ShSp15} and~\cite[Proposition~2.2]{CS19}, the following relations hold between $\nabla^s$ and $\nabla$, and between $\div^s$ and $\div$.

\begin{proposition}\label{prop:Riesz}
For all $\psi\in C^1_c(\R^n)$ we have
\[
\nabla^s\psi(x)=\nabla\mathcal I^{1-s}\psi(x)=\mathcal I^{1-s}\nabla\psi(x)\quad\text{for all $x\in\R^n$}.
\]
For all $\Psi\in C^1_c(\R^n;\R^n)$ we have 
\[
\div^s\Psi(x)=\div\mathcal I^{1-s}\Psi(x)=\mathcal I^{1-s}\div\Psi(x)\quad\text{for all $x\in\R^n$}.
\]
\end{proposition}

For additional properties of the Riesz $s$-fractional gradient and $s$-divergence, we refer the reader to~\cite{CS19,CS23-1,ShSp15,ShSp18,Silhavy20}.


\subsection{Riesz fractional \texorpdfstring{$s$}{s}-variation}

In this subsection, we recall the definition of the \emph{Riesz $s$-fractional variation}, $s\in (0,1)$, originally introduced in~\cite{CS19}, which relies on formula~\eqref{eq:dual}. We begin by introducing the notion of the \emph{weak $s$-fractional gradient}.

\begin{definition}
Let $\Omega\subseteq\R^n$ be an open set, $s\in (0,1)$, $p\in [1,\infty]$, and $u\in L^p(\R^n)$. We say that $u$ has a \emph{weak $s$-fractional gradient} in $\Omega$ if there exists a function $v\in L^1_{\rm loc}(\Omega;\R^n)$ such that
\begin{equation}\label{eq:weak-gradient}
\int_\Omega v(y)\cdot\Psi(y)\,\de y=-\int_{\R^n}u(y)\div^s\Psi(y) \,\de y\quad\text{for all $\Psi\in C^1_c(\Omega;\R^n)$}.
\end{equation}
\end{definition}

By definition, when it exists, the weak $s$-fractional gradient is unique. Moreover, in view of Proposition~\ref{prop:dual}, every function $\psi\in C^1_c(\R^n)$ admits a weak fractional gradient in $\R^n$, which coincides with the Riesz $s$-fractional gradient. Therefore, with a slight abuse of notation, we denote the weak $s$-fractional gradient of $u$ by $\nabla^su$. 

\begin{remark}
We point out that, in order to define the weak $s$-fractional gradient of $u$ in $\Omega$, we need that $u$ is defined on the whole of $\R^n$ and satisfies a suitable integrability condition at infinity. Otherwise, the second integral in~\eqref{eq:weak-gradient} would not be well-defined, since $\div^s \Psi$ does not have compact support.
\end{remark}

As a consequence of~\cite[Proposition~3.2(iii) and Corollary 3.6]{CS23-1}, we have the following result.

\begin{proposition}\label{prop:I1-sDu}
Let $s\in(0,1)$ and $u\in L^\infty(\R^n)$ with $Du\in\mathcal M_b(\R^n;\R^n)$. Then, $u$ has a weak $s$-fractional gradient $\nabla^s u\in L^1_{\rm loc}(\R^n;\R^n)$, and
\[
\nabla^su=\mathcal I^{1-s}Du\quad\text{in $\R^n$}.
\]
Moreover, for all $R>0$ we have
\begin{equation}\label{eq:V1s}
\|\nabla^s u\|_{L^1(B_R)}\le \frac{2\omega_nR^{1-s}\mu_s}{(1-s)2^s}|Du|(B_{3R})+\frac{2^{1+s}\omega_n^2R^{n-s}\mu_s\Gamma(1-s)}{s}\|u\|_{L^\infty(\R^n)}.
\end{equation}
\end{proposition}

\begin{remark}\label{rem:behaviour}
Let us state the behavior of the constants appearing in~\eqref{eq:V1s} as $s\to 1^-$, which will play a crucial role in the proof of the $\Gamma$-limsup inequality. By~\eqref{eq:mus} and the fact that $x\Gamma(x)\to 1$ as $x\to 0^+$ we have
\[
\lim_{s\to 1^-}\frac{2\omega_nR^{1-s}\mu_s}{(1-s)2^s}=n,\qquad \lim_{s\to 1^-}\frac{2^{1+s}\omega_n^2R^{n-s}\mu_s\Gamma(1-s)}{s}=4n\omega_n R^{n-1}.
\]
\end{remark}

We are now in a position to define the notion of \emph{$s$-fractional variation}.

\begin{definition}
Let $\Omega \subseteq\R^n$ be an open set, $s\in (0,1)$, $p\in[1,\infty]$, and $u\in L^p(\R^n)$. We define the \emph{$s$-fractional variation} of $u$ in $\Omega$ as
\[
V^s(u,\Omega) \coloneqq \sup\left\{\int_{\R^n}u(y)\div^s\Psi(y)\,\de y:\Psi\in C^1_c(\Omega;\R^n),\,\|\Psi\|_{L^\infty(\Omega)}\le 1\right\}.
\]
\end{definition}

As in the case of the weak $s$-fractional gradient, the function $u$ must be defined on the whole of $\R^n$ in order to define the $s$-fractional variation of $u$ in $\Omega$. As shown in~\cite[Theorem~3.2]{CS19} for the case $\Omega=\R^n$ and $p=1$ (see also~\cite[Section~1]{CS24} for the general case), one can exploit the Riesz representation theorem to obtain the following relation between the $s$-fractional variation and vector-valued Radon measures.

\begin{proposition}\label{prop:struct}
Let $\Omega \subseteq\R^n$ be an open set, $s\in (0,1)$, $p\in [1,\infty]$, and $u\in L^p(\R^n)$. Then, $u$ has finite $s$-fractional variation in $\Omega$ if and only if there exists a Radon measure $D^su\in\mathcal M_b(\Omega;\R^n)$ satisfying
\begin{equation}\label{eq:struct}
\int_\Omega\Psi(y)\cdot\de D^su(y)=-\int_{\R^n}u(y)\div^s\Psi(y)\,\de y\quad\text{for all $\Psi\in C^1_c(\Omega;\R^n)$}.
\end{equation}
In particular, we have
\[
V^s(u,A)=|D^su|(A)\quad\text{for all open sets $A\subseteq\Omega$}.
\]
\end{proposition}

\begin{remark}
Clearly, a similar result holds true if $u\in L^p(\R^n)$ satisfies
\[
V^s(u,A)<\infty\quad\text{for all open sets $A\subset\subset\Omega$}.
\]
In this case, there exists a Radon measure $D^su\in\mathcal M(\Omega;\R^n)$ such that~\eqref{eq:struct} holds.
\end{remark}

From now on, with a slight abuse of notation, we use $|D^su|(\Omega)$ to denote the $s$-fractional variation of $u$ in $\Omega$. We conclude this subsection by recalling some properties of the $s$-fractional variation. The first one is a simple consequence of Propositions~\ref{prop:dual} and~\ref{prop:struct}. 

\begin{proposition}\label{prop:Ds-nablas}
Let $\Omega \subseteq\R^n$ be an open set, $s\in (0,1)$, $p\in [1,\infty]$, and $u\in L^p(\R^n)$. If $u$ has a weak $s$-fractional gradient $\nabla^su\in L^1(\Omega;\R^n)$, then $u$ has finite $s$-fractional variation in $\Omega$ and
\[
D^s u=\nabla^s u\,\mathcal L^n\quad\text{in $\mathcal M_b(\Omega;\R^n)$}.
\] 
\end{proposition}

The second result, stated in~\cite[Proposition~3.12]{CS23-1}, allows to estimate the $s_0$-fractional variation of a function $u\in L^1(\R^n)$ in terms of its $s$-fractional variation for $s\in (s_0,1)$.

\begin{proposition}\label{prop:bvss0}
Let $s\in (0,1)$. Assume that $u\in L^1(\R^n)$ has finite $s$-fractional variation in $\R^n$. Then, for all $s_0\in(0,s)$, the function $u$ also has finite $s_0$-fractional variation in $\R^n$, and
\begin{align}\label{eq:Vss0}
|D^{s_0}u|(\R^n)\le\frac{s\mu_{1+s_0-s}\omega_n^\frac{s_0}{s}(n+2s_0-s)^\frac{s-s_0}{s}|D^s\mathbf{1}_{B_1}|(\R^n)^\frac{s-s_0}{s}}{s_0(n+s_0-s)(s-s_0)}\|u\|_{L^1(\R^n)}^\frac{s-s_0}{s}|D^su|(\R^n)^\frac{s_0}{s}.
\end{align}
\end{proposition}

\begin{remark}\label{rem:unif}
Let us state the behavior of the constant appearing in~\eqref{eq:Vss0} as $s\to 1^-$, which plays a crucial role in the proof of the compactness result. Since $\mathbf{1}_{B_1}\in BV(\R^n)$, by~\cite[Theorem~4.9]{CS23-1} we have that $\mathbf{1}_{B_1}$ has finite $s$-fractional variation in $\R^n$ and
\[
\lim_{s\to 1^-}|D^s \mathbf{1}_{B_1}|(\R^n)=|D \mathbf{1}_{B_1}|(\R^n)=\mathcal H^{n-1}(\partial B_1)=\omega_n.
\]
Thus, for all $s_0\in (0,1)$, there holds
\[
\lim_{s\to 1^-}\frac{s\mu_{1+s_0-s}\omega_n^\frac{s_0}{s}(n+2s_0-s)^\frac{s-s_0}{s}|D^s\mathbf{1}_{B_1}|(\R^n)^\frac{s-s_0}{s}}{s_0(n+s_0-s)(s-s_0)}=
\frac{\mu_{s_0}\omega_n(n+2s_0-1)^{1-s_0}
}{s_0(n+s_0-1)(1-s_0)}\in(0,\infty).
\]
\end{remark}

We also recall the following result concerning the lower semicontinuity of the $s$-fractional variation as $s\to 1^-$.

\begin{proposition}\label{prop:lsc-Ds}
Let $\Omega\subseteq\R^n$ be an open set and let $(s_k)_k\subset (0,1)$ be such that $s_k\to 1$ as $k\to\infty$. Let $(u_k)_k\subset L^\infty(\R^n)$ and $u\in L^\infty(\Omega)$ satisfy
\begin{equation}\label{eq:L1-loc}
u_k\to u\quad\text{strongly in $L^1(\Omega)$ as $k\to\infty$},\qquad \sup_{k\in\N}\|u_k\|_{L^\infty(\R^n)}<\infty.
\end{equation}
Then,
\[
|Du|(\Omega)\le\liminf_{k\to\infty}|D^{s_k}u_k|(\Omega). 
\]
\end{proposition}

\begin{proof}
The proof follows the one of~\cite[Theorem~4.13(i)]{CS23-1}, where the authors assume that
\[
u_k\to u\quad\text{strongly in $L^1_{\rm loc}(\R^n)$ as $k\to\infty$}, \qquad \sup_{k\in\N}\|u_k\|_{L^\infty(\R^n)}<\infty.
\]
However, a closer inspection of their proof, together with the fact that $|Du|(\Omega)$ depends only on the values of $u$ in $\Omega$, shows that the same argument applies under the weaker assumption~\eqref{eq:L1-loc}.
\end{proof}

We conclude with the following approximation result for functions $u\in L^\infty(\R^n)$ with finite $s$-fractional variation. The proof is similar to that of~\cite[Theorems~3.7 and~3.8]{CS19}, and is postponed to Appendix~\ref{app:a}.

\begin{proposition}\label{prop:app}
Let $\Omega\subseteq\R^n$ be an open set and $s\in (0,1)$. Assume that $u\in L^\infty(\R^n)$ satisfies
\[
|D^su|(A)<\infty\quad\text{for all open sets $A\subset\subset\Omega$}.
\]
Then, there exists a sequence $(u_k)_k\subset C_c^\infty(\R^n)$ such that 
\begin{align*}
&u_k\to u& &\text{strongly in $L^1_{\rm loc}(\R^n)$ for all $p\in [1,\infty)$ as $k\to\infty$},\\
&D^s u_k=\nabla^s u_k\,\mathcal L^n\rightharpoonup D^su& &\text{weakly* in $\mathcal M(\Omega;\R^n)$ as $k\to\infty$}.
\end{align*}
\end{proposition}

\subsection{The functional setting} 
In this subsection, we introduce the functional framework used from Section~\ref{sec:main} on. Let $T\subset\R$ be a collection of $M$ distinct ordered real numbers, that is
\[
T\coloneqq\{c_1,\dots,c_M\}\subset\R\quad\text{with $c_1<c_2<\cdots<c_M$}.
\]
 We recall the definition of the space of $T$-valued functions of bounded variation, see e.g.~\cite{AB90-1,AB90-2}.

\begin{definition}
Let $\Omega\subseteq\R^n$ be an open set. We define 
\begin{align*}
BV(\Omega;T)&\coloneqq \{u\colon\Omega\to T\text{ measurable}: Du\in\mathcal M_b(\Omega;\R^n)\}.
\end{align*}
\end{definition}

We recall some properties of the space $BV(\Omega;T)$ that will be used throughout the paper. For further details, we refer the reader to the monograph~\cite{AFP00}. Let $u\in BV(\Omega;T)$ be fixed. For each $i\in\{1,\dots,M\}$, the set $E_i\coloneq\{x\in\Omega:u(x)= c_i \}$ has finite perimeter in $\Omega$, and we can write
\[
u=\sum_{i=1}^M c_i \mathbf{1}_{E_i}\quad\text{a.e.\ in $\Omega$}.
\]
The jump set of $u$ is defined as
\[
S_u\coloneqq \bigcup_{i=1}^M\partial^* E_i\cap \Omega,
\]
where $\partial^*$ denotes the reduced boundary of a set of finite perimeter. For $\mathcal H^{n-1}$-a.e.\ $x\in S_u$ there exist two different indices $i,j\in\{1,\dots,M\}$ and $\nu\in\mathbb S^{n-1}$ such that
\begin{align*}
&\lim_{r\to 0^+}\frac{\mathcal L^n(B_r(x)\cap E_i)}{\mathcal L^n(B_r(x))}=\lim_{r\to 0^+}\frac{\mathcal L^n(B_r(x)\cap E_j)}{\mathcal L^n(B_r(x))}=\frac{1}{2},\\
&\lim_{r\to 0^+}\frac{\mathcal L^n(\{y\in B_r(x)\cap (\R^n\setminus E_i):(y-x)\cdot \nu>0\})}{\mathcal L^n(B_r(x))}=0,\\
&\lim_{r\to 0^+} \frac{\mathcal L^n(\{y\in B_r(x)\cap (\R^n\setminus E_j):(y-x)\cdot \nu<0\})}{\mathcal L^n(B_r(x))}=0.
\end{align*}
The triplet 
\[
(u^+(x),u^-(x),\nu_u(x)) \coloneqq (c_i,c_j,\nu) 
\]
is uniquely determined for $\mathcal H^{n-1}$-a.e.\ $x\in S_u$, up to a change of sign of $\nu_u(x)$ and an interchange of $u^+(x)$ and $u^-(x)$.

When $\Omega \subset\R^n$ is a bounded open set with a Lipschitz boundary, by~\cite[Lemma~3.3]{AB90-2}, we can derive the following extension property for the space $BV(\Omega;T)$.

\begin{proposition}\label{prop:extension}
Let $\Omega\subset\R^n$ be a bounded open set with Lipschitz boundary and $u\in BV(\Omega;T)$. Then, there exists a function $v\in BV(\R^n;T)$ such that
\[
v=u\quad\text{a.e.\ in $\Omega$},\qquad |Dv|(\partial\Omega)=0.
\]
\end{proposition}

Finally, we define the space of $T$-valued functions with finite $s$-fractional variation in $\Omega$.

\begin{definition}\label{def:BVs}
Let $\Omega\subseteq\R^n$ be an open set and $s\in (0,1)$. We define
\begin{align*}
BV^s(\Omega;T)&\coloneqq \{u\colon\R^n\to T\text{ measurable}: V^s(u,\Omega)<\infty\}.
\end{align*} 
\end{definition}

We point out that every function $u\in BV^s(\Omega;T)$ is defined on the whole of $\R^n$, and that the set $\Omega$ refers only to the domain where the $s$-fractional variation of $u$ is finite.

\section{Main result}\label{sec:main}

In this section, we state the main result of the paper, which is Theorem~\ref{thm:main}. We begin by introducing the class of densities under consideration. Let us fix two constants $0<\lambda\le\Lambda<\infty$. We define $\mathcal G(\lambda, \Lambda)$ as the class of functions $\psi\colon\R^n\times\R^n\to[0,\infty)$ satisfying:
\begin{itemize}
\item[(i)] $\psi$ is continuous on $\R^n \times\R^n$;
\item[(ii)] for every $x\in\R^n$, the map $\xi \mapsto \psi(x,\xi)$ is positively 1-homogeneous and convex;
\item[(iii)] $\lambda|\xi|\le\psi(x,\xi)\le\Lambda|\xi|$ for all $(x,\xi)\in\R^n\times\R^n$.
\end{itemize} 

Let $(\psi_k)_k\subset\mathcal G(\lambda,\Lambda)$ be a fixed sequence. We require the following \emph{uniform approximation condition} for the sequence $(\psi_k)_k$:

\begin{enumerate}[label=(\Alph*), ref=\Alph*]
\item \label{A} For every $\delta>0$, there exists a natural number $N\in\N$, depending only on $\delta$, such that for all $k\in\N$ there exist two families of functions $(b_k^i)_{i=1}^N,(\varphi_k^i)_{i=1}^N$ with the following properties:
\begin{itemize}
\item[(A1)] for all $i\in\{1,\dots,N\}$, the functions $b_k^i\colon\R^n\to[0,\infty)$ are uniformly continuous, bounded, and satisfy
\[
\sup_{k\in\N}\|b_k^i\|_{L^\infty(\R^n)}<\infty;
\]
\item[(A2)] for all $i\in\{1,\dots,N\}$, the functions $\varphi_k^i\colon\R^n\to[0,\infty)$ are positively 1-homogeneous, convex, and satisfy
\[
\sup_{k\in\N}\|\varphi_k^i\|_{C(\mathbb S^{n-1})}<\infty;
\]
\item[(A3)] we have
\[
\left|\psi_k(x,\nu)-\sum_{i=1}^Nb_k^i(x)\varphi_k^i(\nu)\right|\le\delta\quad\text{for all $(x,\nu)\in\R^n\times\mathbb S^{n-1}$ and for all $k\in\N$.}
\]
\end{itemize}
\end{enumerate}
Note that the functions $b_k^i$ and $\varphi_k^i$ in general depend on $\delta$ which is omitted in the notation. We also say that a function $\psi\in\mathcal G(\lambda, \Lambda)$ satisfies the uniform approximation condition~\eqref{A} if the constant sequence $\psi_k\coloneq \psi$, $k \in \mathbb{N}$, satisfies the uniform approximation condition~\eqref{A}. A prototype of a function $\psi\in\mathcal G(\lambda,\Lambda)$ satisfying condition~\eqref{A} is given by
\[
\psi(x,\xi) \coloneqq \sum_{i=1}^N b^i(x)\varphi^i(\xi)\quad \text{for all $(x,\xi)\in\R^n\times\R^n$},
\]
where $b^i\colon\R^n\to[0,\infty)$ are bounded and uniformly continuous, and $\varphi^i\colon\R^n\to[0,\infty)$ are convex and positively 1-homogeneous. We refer the reader to Remark~\ref{rem:A} below for a more detailed discussion of condition~\eqref{A}.

We further fix a sequence $(s_k)_k \subset (0,1)$ such that $s_k\to 1$ as $k\to\infty$. We assume that the sequences $(\psi_k)_k$ and $(s_k)_k$ satisfy the following \emph{compatibility condition}. For given $\delta >0$, let $(b_k^i)_{i=1}^N$, $k\in\N$, be the family of functions given by the uniform approximation condition~\eqref{A}, and let $r_{b_k^i}$ denote the radius of uniform continuity of $b_k^i$ defined in~\eqref{eq:rb}. We assume that
\begin{equation} \label{eq:compatibility}
(1-s_k)\log(r_{b_k^i}(\eta))\to 0\quad\text{as $k\to\infty$ for all $i\in\{1,\dots,N\}$, $\eta>0$, and $\delta >0$.}
\end{equation}
Clearly, this condition is satisfied for a constant sequence satisfying the uniform approximation condition~\eqref{A}. Further clarification on the compatibility condition~\eqref{eq:compatibility} can be found in Remark~\ref{rem:compatibility} below.

Let $\Omega\subset\R^n$ be an open set. For all $s\in (0,1)$ and $u\in L^1(\Omega)$, we define the (possibly empty) set of all extensions $v\in BV^s(\Omega;T)$ of $u$ as
\[
\mathcal A^s(u,\Omega) \coloneqq \left\{v\in BV^s(\Omega;T):v=u\text{ a.e.\ in $\Omega$}\right\}
\]
(recall that a function in $BV^s(\Omega;T)$, introduced in Definition~\ref{def:BVs}, is defined on the whole space $\R^n$). For all $k\in\N$, we consider the functional $\mathcal F_k \colon L^1(\Omega)\to [0,\infty]$, defined as
\begin{equation}\label{eq:Fk}
\mathcal F_k(u) \coloneqq 
\begin{cases}
\displaystyle\inf_{v\in\mathcal A^{s_k}(u,\Omega)}\int_\Omega \psi_k\left(y,\frac{\de D^{s_k}v}{\de |D^{s_k}v|}(y)\right)\de |D^{s_k}v|(y)
&\text{if $\mathcal A^{s_k}(u,\Omega)\neq\emptyset$},\\
\displaystyle\infty&\text{otherwise}.
\end{cases}
\end{equation}
Our goal is to study the $\Gamma$-limit of $(\mathcal F_k)_k$ with respect to the strong topology of $L^1(\Omega)$.

To this end, for all $x\in\R^n$, $ i,j\in\{1,\dots,M\} $, $\nu\in \mathbb S^{n-1}$, $r > 0$, and $k\in\N$ we consider the minimization problem 
\begin{align*}
&m_k( c_i,c_j, Q_r^\nu(x)) \\
&\coloneqq \inf\left\{
\int_{Q_r^\nu(x)}\psi_k\left(y,\frac{\de Du}{\de |Du|}(y)\right)\de|Du|(y):\text{$u\in BV(Q_r^\nu(x);T)$, $u=u_{i,j}^{x,\nu}$ on $\partial Q_r^\nu(x)$}\right\},
\end{align*}
where the function $u_{i,j}^{x,\nu}\colon\R^n\to T$ is defined as
\begin{equation}\label{eq:uij}
u_{i,j}^{x,\nu}(y)\coloneq
\begin{cases}
 c_i &\text{if $(y-x)\cdot \nu>0$},\\
 c_j &\text{if $(y-x)\cdot \nu<0$}.
\end{cases}
\end{equation}
Here, and in the following, the equality $u=u_{i,j}^{x,\nu}$ on $\partial Q_r^\nu(x)$ is always meant to hold in the sense of traces, or by a classical argument, in a neighborhood of $\partial Q_r^\nu(x)$. 

Moreover, we consider the associated cell formulas
\begin{align}
&\psi'(x, c_i,c_j ,\nu)\coloneq \limsup_{r\to 0^+} \liminf_{k\to\infty} \frac{m_k( c_i,c_j ,Q_r^\nu(x))}{r^{n-1}},\label{eq:cell1}\\
&\psi''(x, c_i,c_j ,\nu)\coloneq \limsup_{r\to 0^+} \limsup_{k\to\infty} \frac{m_k( c_i,c_j ,Q_r^\nu(x))}{r^{n-1}}.\label{eq:cell2}
\end{align}

The main result of this paper is the following $\Gamma$-convergence result.

\begin{theorem}\label{thm:main}
Let $\Omega\subset\R^n$ be a bounded open set with Lipschitz boundary. Let $(\psi_k)_k\subset\mathcal G(\lambda,\Lambda)$ and let $(s_k)_k \subset (0,1)$ satisfy $s_k\to 1$. Let $(\mathcal F_k)_k$ be defined as in~\eqref{eq:Fk}.
\begin{itemize}
\item[(1)] {\bf Compactness.} Let $(u_k)_k\subset L^1(\Omega)$ be such that 
\[
\sup_{k\in\N}\mathcal F_k(u_k)<\infty.
\]
Then, there exist a not relabeled subsequence and a function $u\in BV(\Omega;T)$ such that 
\[
u_k\to u\quad\text{strongly in $L^1(\Omega)$ as $k\to\infty$}.
\]
\item[(2)] {\bf $\Gamma$-convergence.} Assume that $(\psi_k)_k$ satisfies the uniform approximation condition~\eqref{A} and that $(s_k)_k$ satisfies the compatibility condition~\eqref{eq:compatibility}. Then, there exists a not relabeled subsequence such that for all $x\in \Omega$, $i,j\in\{1,\dots,M\}$, and $\nu\in \mathbb S^{n-1}$ 
\[
\psi'(x, c_i,c_j ,\nu)=\psi''(x, c_i,c_j ,\nu)\eqqcolon \psi_0 (x, c_i,c_j ,\nu),
\]
where $\psi'$ and $\psi''$ are computed along this subsequence. Moreover, for this subsequence, the sequence of functionals $(\mathcal F_k)_k$ $\Gamma$-converges with respect to the strong topology of $L^1(\Omega)$ to the functional $\mathcal F_0 \colon L^1(\Omega)\to [0,\infty]$, defined as
\begin{equation}\label{eq:Fhom}
\mathcal F_0 (u) \coloneqq
\begin{cases}
\displaystyle\int_{S_u \cap \Omega} \psi_0 \big(y, u^+(y), u^-(y), \nu_u(y)\big)\, \de \mathcal H^{n-1}(y) & \text{if $u\in BV(\Omega;T)$},\\
\infty & \text{otherwise}.
\end{cases}
\end{equation}
\end{itemize}
\end{theorem}

\begin{remark}
The compactness result in Theorem~\ref{thm:main}(1) does not require the uniform approximation condition~\eqref{A} or the compatibility condition~\eqref{eq:compatibility}. In particular, these conditions are needed only in the proof of the $\Gamma$-limsup inequality, while the $\Gamma$-liminf inequality holds for any sequences $(\psi_k)_k\subset\mathcal G(\lambda,\Lambda)$ and $(s_k)_k\subset (0,1)$ with $s_k\to 1$ as $k\to\infty$.
\end{remark}

The proof of Theorem~\ref{thm:main}(1) is presented in Section~\ref{sec:comp}, while that of Theorem~\ref{thm:main}(2) is given in Section~\ref{sec:Gamma}. As a consequence of Theorem~\ref{thm:main} and the Urysohn property of $\Gamma$-convergence, see~\cite[Proposition~8.3]{DM93}, we deduce the following corollary.

\begin{corollary}\label{coro:Gamma-conv}
Let $\Omega\subset\R^n$ be a bounded open set with Lipschitz boundary. Let $(\psi_k)_k\subset\mathcal G(\lambda,\Lambda)$ satisfy the uniform approximation condition~\eqref{A} and let $(s_k)_k \subset (0,1)$ satisfy $s_k\to 1$ as $k\to\infty$ and the compatibility condition~\eqref{eq:compatibility}. Let $(\mathcal F_k)_k$ be defined as in~\eqref{eq:Fk}. Assume that for all $x\in \Omega$, $i,j\in\{1,\dots,M\}$, and $\nu\in \mathbb S^{n-1}$ we have 
\[
\psi'(x, c_i,c_j ,\nu)=\psi''(x, c_i,c_j ,\nu) \eqqcolon \psi_0 (x, c_i,c_j ,\nu).
\]
Then, $(\mathcal F_k)_k$ $\Gamma$-converges with respect to the strong topology of $L^1(\Omega)$ to the functional $\mathcal F_0 $ defined by~\eqref{eq:Fhom}.
\end{corollary}

\begin{remark}[The uniform approximation condition~\eqref{A}]\label{rem:A}
\begin{itemize}
\item[(a)] A \emph{prototype example} of a sequence $(\psi_k)_k\subset\mathcal G(\lambda,\Lambda)$ arises in the context of homogenization, where
\begin{equation}\label{eq:psik}
\psi_k(x,\xi) \coloneqq \psi\left(\frac{x}{\varepsilon_k},\xi\right)\quad\text{for all $(x,\xi)\in\R^n\times\R^n$},
\end{equation}
with $\psi\in\mathcal G(\lambda,\Lambda)$ and $(\varepsilon_k)_k\subset(0,\infty)$ a sequence satisfying $\varepsilon_k\to 0$. In this setting, the sequence $(\psi_k)_k$ satisfies the uniform approximation condition~\eqref{A} if and only if $\psi$ satisfies the approximation condition~\eqref{A}.
 
\item[(b)] The approximation condition~\eqref{A} holds when $\psi\in\mathcal G(\lambda,\Lambda)$ is periodic in $x\in\R^n$ with respect to a fixed full-rank lattice $L\subset\R^n$. Indeed, in this case, $\psi$ can be seen as a function defined on $\mathbb T^n(L)\times\R^n$, where $\mathbb T^n(L)$ denotes the $n$-dimensional torus $\R^n/L$.
Since $\psi$ is uniformly continuous on $\mathbb T^n(L) \times \mathbb S^{n-1}$, for every $\delta>0$ there exists $r_\delta>0$ such that
\[
|\psi(x,\nu)-\psi(y,\nu)|<\delta\quad \text{for all $x,y\in\mathbb T^n(L)$ with $|x-y|<r_\delta$ and $\nu\in\mathbb S^{n-1}$}.
\]
Let $(B_{r_\delta}(x^i))_{i=1}^N$ be an open cover of $\mathbb T^n(L)$, and let $(b^i)_{i=1}^N$ be a partition of unity subordinate to this cover. Define
\[
\varphi^i(\xi) \coloneqq \psi(x^i,\xi)\quad \text{for all $\xi\in\R^n$ and $i\in\{1, \dots,N\}$}.
\]
Then $\psi$ satisfies the approximation condition~\eqref{A} with families $(b^i)_{i=1}^N, (\varphi^i)_{i=1}^N$ since
\[
\left|\psi(x,\nu)-\sum_{i=1}^N b^i(x) \varphi^i(\nu)\right|\le \sum_{i=1}^N b^i(x) \left|\psi(x,\nu)-\psi(x^i,\nu)\right|\le \delta
\]
for all $(x,\nu)\in\mathbb T^n(L)\times\mathbb S^{n-1}$.

\item[(c)] With a similar argument as in (b), we can also consider functions $\psi\in\mathcal G(\lambda,\Lambda)$ which are \emph{almost periodic in $x\in\R^n$ uniformly in $\xi\in\R^n$} in the following sense: there exists a sequence $(\psi_h)_h\subset \mathcal G(\lambda,\Lambda)$, with each $\psi_h$ periodic in $x\in\R^n$ with respect to a fixed lattice $L_h$ (possibly depending on $h$), such that 
\[
\|\psi-\psi_h\|_{C(\R^n\times \mathbb S^{n-1})}\to 0\quad\text{as $h\to\infty$}.
\]
\end{itemize}

\end{remark}

\begin{remark}[The compatibility condition~\eqref{eq:compatibility}]\label{rem:compatibility}
\begin{itemize}
\item[(a)] In the homogenization setting described in Remark~\ref{rem:A}(a), if the function $\psi\in\mathcal G(\lambda,\Lambda)$ satisfies the approximation condition~\eqref{A} with families $(b^i)_{i=1}^N,(\varphi^i)_{i=1}^N$, then one can define the family $(b^i_k)_{i=1}^N$ as $b^i_k(x) \coloneqq b^i(x/\varepsilon_k)$ for $x\in\R^n$, $k\in\N$, and $i\in\{1,\dots,N\}$, and thus
\[
r_{b_k^i}(\eta)=\varepsilon_k r_{b^i}(\eta)\quad\text{for all $k\in\N$, $i \in \{1,\dots,N\}$, and $\eta>0$}.
\]
In this case, the compatibility condition~\eqref{eq:compatibility} is equivalent to requiring
\begin{equation}\label{eq:compatibility2}
(1-s_k)\log \varepsilon_k \to 0 \quad \text{as $k \to \infty$}.
\end{equation}
This condition ensures that the sequence $(\varepsilon_k)_k$ does not vanish too quickly in comparison to the rate at which $(s_k)_k$ converges to 1. 
\item[(b)] We point out that conditions similar to~\eqref{eq:compatibility2} appear in the literature in the context of fractional gradients and homogenization. For example, a similar assumption is made in~\cite[Eq.~(3.1)]{ACFS25}, in the context of $\Gamma$-convergence of discrete dislocation fractional energies, where the roles of $1-s_k$ and $\varepsilon_k$ are played by parameters $\alpha$ and $\rho_\alpha$, respectively. Furthermore, in~\cite[Eq.~(8)]{BBD24}, concerning the homogenization of quadratic fractional energies of Gagliardo-type, the authors require
\[
\frac{1-s_k}{\varepsilon_k^2} \to 0 \quad \text{as $k \to \infty$},
\]
which in particular implies~\eqref{eq:compatibility2}.
\end{itemize}
\end{remark}

Let us focus on the homogenization case introduced already in Remark~\ref{rem:A}(a). For all $x\in\R^n$, $ i,j\in\{1,\dots,M\} $, $\nu\in\mathbb S^{n-1}$, and $r>0$, we define
\begin{align*}
& m( c_i,c_j ,Q_r^\nu(x)) \\
&\coloneqq \inf \left\{\int_{Q_r^\nu(x)}\psi \left(y,\frac{\de Du}{\de |Du|}(y)\right)\de|Du|(y):\text{$u\in BV(Q_r^\nu(x);T)$, $u=u_{i,j}^{x,\nu}$ on $\partial Q_r^\nu(x)$}\right\},
\end{align*}
where $u_{i,j}^{x,\nu}$ is the function defined in~\eqref{eq:uij}.

\begin{lemma}
Let $\psi\in\mathcal G(\lambda,\Lambda)$ and let $(\varepsilon_k)_k \subset (0,\infty)$ be such that $\varepsilon_k\to 0$ as $k\to\infty$. Assume that for all $x\in\R^n$, $ i,j\in\{1,\dots,M\} $, and $\nu\in \mathbb S^{n-1}$ the following limit exists and it is independent of $x$:
\begin{equation}\label{eq:AP}
\lim_{t\to\infty}\frac{m( c_i,c_j ,Q_t^\nu(tx))}{t^{n-1}} \eqqcolon \psi_{\rm hom} ( c_i,c_j ,\nu).
\end{equation}
Then, for all $r>0$ and $x\in\R^n$ it holds that
\[
\lim_{k\to\infty}\frac{m_k( c_i,c_j ,Q_r^\nu(x))}{r^{n-1}}= \psi_{\rm hom} ( c_i,c_j ,\nu).
\]
In particular, for all $x\in\R^n$, $i,j\in\{1,\dots,M\}$, and $\nu\in \mathbb S^{n-1}$ we have
\[
\psi'(x, c_i,c_j ,\nu)=\psi''(x, c_i,c_j ,\nu)= \psi_{\rm hom} ( c_i,c_j ,\nu).
\]
\end{lemma}

\begin{proof}
Let $u\in BV(Q_r^\nu(x);T)$ be such that $u=u_{i,j}^{x,\nu}$ on $\partial Q_r^\nu(x)$. We define
\[
u_k(y) \coloneqq u(\varepsilon_k y)\quad\text{for $y\in Q_{r/\varepsilon_k}^\nu(x/\varepsilon_k)$}.
\]
Then, $u_k\in BV(Q_{r/\varepsilon_k}^\nu(x/\varepsilon_k);T)$ and $u_k=u_{i,j}^{x,\nu}$ on $\partial Q_{r/\varepsilon_k}^\nu(x/\varepsilon_k)$. By the change of variable $y'=y/\varepsilon_k$ we have
\begin{align*}
\int_{Q_r^\nu(x)}\psi\left(\frac{y}{\varepsilon_k},\frac{\de Du}{\de |Du|}(y)\right)\de|Du|(y)&=\varepsilon_k^{n-1}\int_{Q_{r/\varepsilon_k}^\nu(x/\varepsilon_k)}\psi\left(y',\frac{\de Du_k}{\de |Du_k|}(y')\right)\de|Du_k|(y').
\end{align*}
Hence,
\begin{equation*}
m_k( c_i,c_j ,Q_r^\nu(x))=\varepsilon_k^{n-1}m\left( c_i,c_j ,Q_\frac{r}{\varepsilon_k}^\nu\left(\frac{x}{\varepsilon_k}\right)\right),
\end{equation*}
which gives
\[
\lim_{k\to\infty}\frac{m_k( c_i,c_j ,Q_r^\nu(x))}{r^{n-1}}=\lim_{k\to\infty}\left(\frac{\varepsilon_k}{r}\right)^{n-1}m\left( c_i,c_j ,Q_\frac{r}{\varepsilon_k}^\nu\left(\frac{r}{\varepsilon_k}\frac{x}{r}\right)\right)=\psi_{\rm hom}( c_i,c_j ,\nu),
\]
and concludes the proof.
\end{proof}

\begin{example}
A simple case in which condition~\eqref{eq:AP} holds is when $\psi\in\mathcal G(\lambda,\Lambda)$ is periodic in $x\in\R^n$ with respect to a fixed full-rank lattice $L$, see~\cite[Eq.~(11)]{BDV96}. Moreover, as observed in Remark~\ref{rem:A}(b), in this case $\psi$ satisfies the approximation condition~\eqref{A}. Hence, if $(\varepsilon_k)_k$ satisfies the compatibility condition~\eqref{eq:compatibility2}, we may apply Corollary~\ref{coro:Gamma-conv} to deduce that the sequence $(\mathcal F_k)_k$, defined by~\eqref{eq:Fk} with densities $(\psi_k)_k$ as in~\eqref{eq:psik}, $\Gamma$-converges with respect to the strong topology of $L^1(\Omega)$ to the functional $\mathcal F_0$ given by~\eqref{eq:Fhom}, where the limit density $\psi_0=\psi_{\rm hom}$ is given by formula~\eqref{eq:AP}.
\end{example}

\section{Compactness}\label{sec:comp}

This section is devoted to the proof of Theorem~\ref{thm:main}(1). We begin by showing the following result, which allows us to pass from uniform bounds on the fractional variation in $\Omega$ to uniform bounds on the fractional variation in $\R^n$.

\begin{lemma}\label{lem:bvomega-rn2}
Let $\Omega\subseteq\R^n$ be an open set, $s\in (0,1)$, $p\in [1,\infty]$, and $u\in L^p(\R^n)$. For all $\psi\in C^1_c(\Omega)$ with $\|\psi\|_{L^\infty(\Omega)}\le 1$ we have
\begin{equation}\label{eq:loc-Ds-est}
|D^s(u\psi)|(\R^n)\le |D^su|(\Omega)+\frac{6\omega_n\mu_s}{s(1-s)2^s}\|u\|_{L^p(\R^n)}\|\psi\|^{1-s}_{L^{p'}(\R^n)}\|\nabla\psi\|^s_{L^{p'}(\R^n)}.
\end{equation}
\end{lemma}

\begin{proof}
We fix $\Psi\in C^1_c(\R^n;\R^n)$ with $\|\Psi\|_{L^\infty(\R^n)}\le 1$. By Propositions~\ref{prop:gl-est1},~\ref{prop:gl-est2},~\ref{prop:nonlocal-leibniz}, and H\"older's inequality, we have
\begin{align*}
&\int_{\R^n} u(y) \psi(y)\div^s\Psi(y)\,\de y\\
&=\int_{\R^n}u(y) \div^s(\psi\Psi)(y)\,\de y-\int_{\R^n}u(y) \nabla^s \psi(y)\cdot\Psi(y)\,\de y-\int_{\R^n}u(y)\div^s_{\rm NL}(\Psi,\psi)(y)\,\de y\\
&\le\int_{\R^n}u(y) \div^s(\psi\Psi)(y)\,\de y+\frac{6\omega_n\mu_s}{s(1-s)2^s}\|u\|_{L^p(\R^n)}\|\psi\|^{1-s}_{L^{p'}(\R^n)}\|\nabla \psi\|^s_{L^{p'}(\R^n)}\\
&\le |D^su|(\Omega)+\frac{6\omega_n\mu_s}{s(1-s)2^s}\|u\|_{L^p(\R^n)}\|\psi\|^{1-s}_{L^{p'}(\R^n)}\|\nabla \psi\|^s_{L^{p'}(\R^n)}.
\end{align*}
By taking the supremum over all $\Psi\in C^1_c(\R^n;\R^n)$ with $\|\Psi\|_{L^\infty(\R^n)}\le 1$ we obtain~\eqref{eq:loc-Ds-est}.
\end{proof}

We are now in a position to prove our compactness result. We show the following result which will readily imply Theorem~\ref{thm:main}(1).

\begin{proposition}\label{prop:comp}
Let $\Omega\subseteq\R^n$ be an open set and let $(s_k)_k\subset (0,1)$ be such that $s_k\to 1$ as $k\to\infty$. Let $(u_k)_k\subset L^\infty(\R^n)$ satisfy
\[
\sup_{k\in\N}\left(\|u_k\|_{L^\infty(\R^n)}+|D^{s_k}u_k|(A)\right)<\infty\quad\text{for all open sets $A\subset\subset\Omega$}.
\]
Then, there exist a not relabeled subsequence and a function $u\in L^1(\Omega) \cap BV_{\rm loc}(\Omega)$ such that 
\[
u_k\to u\quad\text{strongly in $L^1(\Omega)$ as $k\to\infty$}.
\]
\end{proposition}

\begin{proof}
Let $(\Omega_j)_j$ be a sequence of open bounded sets with smooth boundaries such that
\[
\Omega_j\subset\subset\Omega_{j+1}\quad\text{for all $j\in\N$},\qquad \bigcup_{j\in\N}\Omega_j=\Omega,
\]
and let $\psi_j\in C^1_c(\Omega_{j+1})$, $j\in\N$, satisfy
\[
0\le \psi_j\le 1\quad\text{in $\Omega_{j+1}$},\qquad \psi_j=1\quad\text{in a neighborhood of $\Omega_j$}. 
\]
By using Lemma~\ref{lem:bvomega-rn2} with $p=\infty$, for all $k\in\N$ we have
\begin{equation}\label{eq:ukpsij-est}
|D^{s_k}(u_k\psi_j)|(\R^n)\le |D^{s_k}u_k|(\Omega_{j+1})+\frac{6\omega_n\mu_{s_k}}{s_k(1-s_k)2^{s_k}}\|u_k\|_{L^\infty(\R^n)}\|\psi_j\|_{L^1(\R^n)}^{1-s_k}\|\nabla\psi_j\|_{L^1(\R^n)}^{s_k}.
\end{equation}

We fix $s_0\in (0,1)$. By Proposition~\ref{prop:bvss0}, there exists $k_0\in\N$ such that $u_k$ has finite $s_0$-fractional variation in $\R^n$ for all $k \ge k_0$. Moreover, by using also Young's inequality with exponents $p = \frac{s_k}{s_0}$ and $p' = \frac{s_k}{s_k-s_0}$, Remark~\ref{rem:unif}, and the estimate~\eqref{eq:ukpsij-est}, we can find a constant $C > 0$, independent of $k\in\N$, such that for all $k \ge k_0$
\begin{align*}
|D^{s_0}(u_k\psi_j)|(\R^n)&\le C\left(\frac{s_k-s_0}{s_k}\|u_k\psi_j\|_{L^1(\R^n)}+\frac{s_0}{s_k}|D^{s_k}(u_k\psi_j)|(\R^n)\right)\\
&\le C\left(\frac{s_k-s_0}{s_k}\|u_k\psi_j\|_{L^1(\R^n)}+\frac{s_0}{s_k}|D^{s_k}u_k|(\Omega_{j+1})\right)\\
&\quad+C\frac{6s_0\omega_n\mu_{s_k}}{s_k^2(1-s_k)2^{s_k}}\|u_k\|_{L^\infty(\R^n)}\|\psi_j\|_{L^1(\R^n)}^{1-s_k}\|\nabla \psi_j\|_{L^1(\R^n)}^{s_k}.
\end{align*}
Thus, thanks to~\eqref{eq:mus}, we can find a constant $C_j > 0$, independent of $k\in\N$, such that for all $k \ge k_0$
\begin{align*}
|D^{s_0}(u_k\psi_j)|(\R^n)&\le C_j\left(\|u_k\|_{L^\infty(\R^n)}+|D^{s_k}u_k|(\Omega_{j+1})\right).
\end{align*}
We can then apply~\cite[Theorem~3.16]{CS19} to obtain the existence of a subsequence $(k_{h,j})_h$ and a function $u_j\in L^1(\R^n)$ satisfying
\[
u_{k_{h,j}}\psi_j\to u_j\quad\text{strongly in $L^1_{\rm loc}(\R^n)$ as $h\to\infty$}.
\]
In particular, we get
\[
u_{k_{h,j}}\to u_j\quad\text{strongly in $L^1(\Omega_j)$ as $h\to\infty$}. 
\]
Moreover, by Proposition~\ref{prop:lsc-Ds} we have
\begin{align*}
|Du_j|(\Omega_j)&\le \liminf_{h\to\infty}| D^{s_{k_{h,j}}}u_{k_{h,j}} |(\Omega_j)<\infty,
\end{align*}
which gives that $u_j\in BV(\Omega_j)$. By a diagonal argument, we can find a subsequence $(k_h)_h\subset\N$ and function $u\in BV_{\rm loc}(\Omega)$ such that $u_{k_h}\to u$ strongly in $L^1_{\rm loc}(\Omega)$ as $h\to\infty$. Finally, since $\sup_{k\in\N}\|u_k\|_{L^\infty(\R^n)}<\infty$, we deduce that $u\in L^1(\Omega)$ and $u_k\to u$ strongly in $L^1(\Omega)$ as $k\to\infty$.
\end{proof}

We can finally prove Theorem~\ref{thm:main}(1).

\begin{proof}[Proof of Theorem~\ref{thm:main}(1)]
Let $(u_k)_k\subset L^1(\Omega)$ be such that 
\[
\sup_{k\in\N}\mathcal F_k(u_k)<\infty.
\]
Then, since $(\psi_k)_k\subset\mathcal G(\lambda,\Lambda)$, for all $k\in\N$ there exists $v_k\in \mathcal A^{s_k}(u_k,\Omega)$ such that 
\[
\|v_k\|_{L^\infty(\R^n)}+|D^{s_k}v_k|(\Omega)\le \max_{i\in\{1,\dots,M\}}|c_i| +\frac{1}{\lambda}\left(\mathcal F_k(u_k)+1\right).
\]
By Proposition~\ref{prop:comp}, we can find a not relabeled subsequence and $u\in L^1(\Omega) \cap BV_{\rm loc}(\Omega)$ satisfying
\begin{equation}\label{eq:conv}
v_k=u_k\to u\quad\text{strongly in $L^1(\Omega)$ as $k\to\infty$}.
\end{equation}
Moreover, by Proposition~\ref{prop:lsc-Ds} we derive that 
\begin{equation}\label{eq:lsc}
|Du|(\Omega)\le \liminf_{k\to\infty}|D^{s_k}v_k|(\Omega)<\infty.
\end{equation}
By combining~\eqref{eq:conv} and~\eqref{eq:lsc}, we conclude that $u\in BV(\Omega;T)$ and that $u_k\to u$ strongly in $L^1(\Omega)$ as $k\to\infty$.
\end{proof}

\section{\texorpdfstring{$\Gamma$}{Gamma}-convergence}\label{sec:Gamma}

In this section, we prove the $\Gamma$-convergence result stated in Theorem~\ref{thm:main}(2). To this end, we first prove the lower bound in Proposition~\ref{prop:lower-bound}, and then the upper bound in Proposition~\ref{prop:upper-bound}. 

We begin by recalling the following $\Gamma$-convergence result for the local case. Let $\Omega\subseteq\R^n$ be an open set and let $\mathcal A(\Omega)$ denote the collection of all open subsets $A\subseteq\Omega$. We consider the family of localized functionals $\mathcal E_k \colon L^1(\Omega)\times \mathcal A(\Omega)\to [0,\infty]$, defined as
\begin{equation}\label{eq:Ek}
\mathcal E_k(u,A) \coloneqq 
\begin{cases}
\displaystyle\int_A \psi_k\left(y,\frac{\de Du}{\de |Du|}(y)\right)\,\de |Du|(y)
&\text{if $A\in\mathcal A(\Omega)$ and $u\in BV(A;T)$},\\
\displaystyle\infty&\text{otherwise}.
\end{cases}
\end{equation}

\begin{theorem}\label{thm:local-gamma}
Let $(\psi_k)_k \subset \mathcal G(\lambda,\Lambda)$. Then, there exists a not relabeled subsequence such that for all $x\in\R^n$, $i,j\in\{1,\dots,M\}$, and $\nu\in \mathbb S^{n-1}$ we have 
\begin{equation}\label{eq:psi-local}
\psi'(x, c_i,c_j ,\nu) = \psi''(x, c_i,c_j ,\nu) \eqqcolon \psi_0 (x, c_i,c_j ,\nu),
\end{equation}
where $\psi'$ and $\psi''$ are computed along this subsequence. Moreover, for this subsequence, given an open and bounded set $\Omega \subset\R^n$ with Lipschitz boundary, the sequence of functionals $(\mathcal E_k(\,\cdot\,,\Omega))_k$ defined by~\eqref{eq:Ek} $\Gamma$-converges with respect to the strong topology of $L^1(\Omega)$ to the functional $\mathcal F_0 (\,\cdot\,, \Omega)$, where the localized functional $\mathcal F_0 \colon L^1(\Omega)\times \mathcal A(\Omega)\to [0,\infty]$ is defined by
\[
\mathcal F_0 (u,A) \coloneqq
\begin{cases}
\displaystyle\int_{S_u \cap A} \psi_0 \big(y, u^+(y), u^-(y), \nu_u(y)\big)\, \de \mathcal H^{n-1}(y) & \text{if $A\in\mathcal A(\Omega)$ and $u\in BV(A;T)$},\\
\infty & \text{otherwise}.
\end{cases}
\]
\end{theorem}

In the case of the homogenization of a density $\psi\in\mathcal G(\lambda,\Lambda)$ that is periodic in $x\in\R^n$ with respect a fixed lattice $L$, a proof of Theorem~\ref{thm:local-gamma} can be found in~\cite[Theorem~4.2]{AB90-2}. In the general (nonperiodic) case, we refer to~\cite[Theorem~2.3, Lemma~6.3, and Lemma~7.5]{FS20}, where this result is proved in the more general setting of functionals defined on piecewise rigid functions. The only difference is that the cell formulas~\eqref{eq:cell1} and~\eqref{eq:cell2} are defined using balls instead of cubes, and that in the present setting the analog of the estimate in~\cite[Lemma~7.5]{FS20} holds in any space dimension and is much simpler as compactness in the space of piecewise constant function with values in $T$ is obtained in a straightforward way. The existence of a subsequence for which~\eqref{eq:psi-local} holds for all $x\in\R^n$ can be obtained by applying~\cite[Theorem~2.3]{FS20} with $\Omega=B_h$, for $h\in\N$, and then using a diagonal argument. We point out that, by construction, for all $x\in\R^n$, $i,j\in\{1,\dots,M\}$, and $\nu\in \mathbb S^{n-1}$ we have 
\[
\psi_0(x, c_i,c_j ,\nu)=\psi_0(x, c_j,c_i ,-\nu)=\hat\psi_0(x, (c_i-c_j) \nu),
\]
for a function $\hat\psi_0\colon\R^n\times\R^n\to \R$ which is positively $1$-homogeneous in the second variable. Moreover, thanks to~\cite[Theorem~5.11]{AFP00} (see also~\cite[Theorem~5.14]{AFP00} and~\cite[Proposition~3.3]{EKM}) the function $\hat\psi_0$ is convex in the second variable. By arguing as in~\cite[Lemma~A.7]{CDMSZ}, we further obtain that
\begin{equation}\label{eq:psi0}
\lambda|\xi|\le \hat\psi_0(x,\xi)\le \Lambda|\xi|\quad\text{for all $(x,\xi)\in\R^n\times\R^n$}.
\end{equation}
In particular, if $\hat\psi_0$ is also continuous on $\R^n\times \R^n$, then $\hat\psi_0\in\mathcal G(\lambda,\Lambda)$.

From now on, we fix the subsequence given by Theorem~\ref{thm:local-gamma} for which~\eqref{eq:psi-local} holds. 

\subsection{\texorpdfstring{$\Gamma$}{Gamma}-liminf inequality} 

We begin by recalling the following classical result for the space of functions of bounded variation on an open bounded set $\Omega\subset\R^n$ with Lipschitz boundary. Since the proof is standard, it is postponed to Appendix~\ref{app:a}.

\begin{lemma}\label{lem:BVO-comp}
Let $\Omega\subset\R^n$ be a bounded open set.
\begin{itemize}
\item[(1)] Assume that $\Omega$ has a Lipschitz boundary. Let $(w_k)_k\subset BV(\Omega)$ be such that 
\[
\sup_{k\in\N}|D w_k|(\Omega)<\infty. 
\]
Then, there exist a not relabeled subsequence, a sequence $(a_k)_k\subset\R$, and a function $w\in BV(\Omega)$ satisfying
\begin{align}
&w_k-a_k\to w & &\text{strongly in $L^1(\Omega)$ as $k\to\infty$},\label{eq:BVcon1}\\
&Dw_k\rightharpoonup Dw & &\text{weakly* in $\mathcal M_b(\Omega;\R^n)$ as $k\to\infty$}.\label{eq:BVcon2}
\end{align}
\item[(2)] Let $(w_k)_k\subset BV_{\rm loc}(\Omega)$ be such that 
\[
\sup_{k\in\N}|D w_k|(A)<\infty\quad\text{for all open sets $A\subset\subset \Omega$}. 
\]
Then, there exist a not relabeled subsequence and a function $w\in BV_{\rm loc}(\Omega)$ satisfying
\begin{align}
Dw_k\rightharpoonup Dw \quad\text{weakly* in $\mathcal M(\Omega;\R^n)$ as $k\to\infty$}.\label{eq:BVcon3}
\end{align}
\end{itemize}
\end{lemma}

As observed in Proposition~\ref{prop:Riesz}, given $s \in (0,1)$ and a function $u\in C_c^1(\R^n)$, there exists a function $v\in C^1(\R^n)$, given explicitly by $v = \mathcal I^{1-s} u$, such that
\begin{equation}\label{eq:nablas-nabla}
\nabla^s u = \nabla v\quad\text{in $\R^n$}.
\end{equation}
This observation plays a central role in proving $\Gamma$-convergence results for energies involving the fractional gradient, by exploiting the known $\Gamma$-convergence results for analogous energies depending on the classical gradient, see e.g.~\cite{KS22}. We now want to show that property~\eqref{eq:nablas-nabla} can be extended to the setting of the $s$-fractional variation.

\begin{proposition}\label{prop:rep}
Let $\Omega\subset\R^n$ be a bounded open set with Lipschitz boundary. Assume that $u\in L^\infty(\R^n)$ has finite $s$-fractional variation in $\Omega$. Then, there exists a function $w\in BV(\Omega)$ such that 
\[
Dw=D^s u\quad\text{in $\mathcal M_b(\Omega;\R^n)$}.
\]
\end{proposition}

\begin{proof}
Let $u \in L^\infty(\R^n)$ be a function with finite $s$-fractional variation in $\Omega$, and let $(u_k)_k \subset C_c^\infty(\R^n)$ be the sequence provided by Proposition~\ref{prop:app}. We define the sequence of functions
\[
w_k \coloneqq \mathcal I^{1-s}u_k,\quad k\in\N,
\]
where $\mathcal I^{1-s}$ denotes the Riesz potential of order $1-s$.

Thanks to Propositions~\ref{prop:Riesz} and~\ref{prop:Ds-nablas} we have 
\[
D w_k=D^s u_k\rightharpoonup D^su\quad\text{weakly* in $\mathcal M(\Omega;\R^n)$ as $k\to\infty$}.
\]
By the Banach-Steinhaus theorem, it follows that
\[
\sup_{k\in\N} |D w_k|(A) <\infty \quad \text{for all open sets $A \subset\subset \Omega$}.
\]
Hence, we can apply Lemma~\ref{lem:BVO-comp}(2) to obtain the existence of a not relabeled subsequence and a function $w\in BV_{\rm loc}(\Omega)$ such that 
\[
Dw_k\rightharpoonup Dw\quad\text{weakly* in $\mathcal M(\Omega;\R^n)$ as $k\to\infty$}.
\]
Hence,
\[
Dw =D^s u\quad\text{in $\mathcal M_{b}(\Omega;\R^n)$},
\]
which yields $|Dw|(\Omega)<\infty$. Since $\Omega\subset\R^n$ is a bounded open set with Lipschitz boundary, by Poincare's inequality we conclude that $w\in BV(\Omega)$.
\end{proof}

As a consequence of Proposition~\ref{prop:comp} and Proposition~\ref{prop:rep}, we derive the following result, which will be used in the proof of the $\Gamma$-liminf inequality.

\begin{lemma}\label{lem:comp2}
Let $\Omega\subset\R^n$ be a bounded open set with Lipschitz boundary. Let $(s_k)_k\subset (0,1)$ be such that $s_k\to 1$ as $k\to\infty$. Assume that $(u_k)_k\subset L^\infty(\R^n)$ satisfies
\[
\sup_{k\in\N}\left(\|u_k\|_{L^\infty(\R^n)}+|D^{s_k}u_k|(\Omega)\right)<\infty.
\]
Then, there exist a not relabeled subsequence and a sequence $(w_k)_k \subset BV(\Omega)$ such that 
\begin{align*}
&w_k-u_k\to 0& &\text{strongly in $L^1(\Omega)$ as $k\to\infty$},\\
& Dw_k=D^{s_k}u_k& &\text{in $\mathcal M_b(\Omega;\R^n)$ for all $k\in\N$}.
\end{align*}
\end{lemma}

\begin{proof}
By Proposition~\ref{prop:comp} there exist a not relabeled subsequence and a function $u\in BV_{\rm loc}(\Omega)$ such that 
\[
u_k\to u\quad\text{strongly in $L^1(\Omega)$ as $k\to\infty$}.
\] 
Moreover, by Proposition~\ref{prop:lsc-Ds} we have
\[
|Du|(\Omega)\le \liminf_{k\to\infty}|D^{s_k} u_k|(\Omega),
\]
which gives that $u\in BV(\Omega)$. In view of Proposition~\ref{prop:rep}, for all $k\in\N$ there exists a function $v_k\in BV(\Omega)$ such that 
\[
Dv_k=D^{s_k}u_k\quad\text{in $\mathcal M_b(\Omega;\R^n)$}.
\]
By applying Lemma~\ref{lem:BVO-comp}(1) to the sequence $(v_k)_k$, we can find a not relabeled subsequence, a sequence $(a_k)_k\subset\R$, and a function $v\in BV(\Omega)$ satisfying
\[
v_k-a_k\to v\quad\text{strongly in $L^1(\Omega)$ as $k\to\infty$},\qquad Dv_k\rightharpoonup Dv\quad\text{weakly* in $\mathcal M_b(\Omega;\R^n)$ as $k\to\infty$}.
\]
Let us fix $\Psi\in C^\infty_c(\Omega;\R^n)$. By using~\cite[Proposition~4.4]{CS23-1}, as $k\to\infty$ we get
\begin{align*}
&\left|\int_{\R^n}u_k(y)\div^{s_k}\Psi(y)\,\de y-\int_\Omega u(y)\div\Psi(y)\,\de y\right|\\
&\le\|\div^{s_k}\Psi-\div\Psi\|_{L^1(\R^n)} \sup_{k\in\N}\|u_k\|_{L^\infty(\R^n)}+\|\div\Psi\|_{L^\infty(\Omega)}\|u_k-u\|_{L^1(\supp\Psi)}\to 0.
\end{align*}
Hence,
\begin{align*}
\int_\Omega u(y)\div\Psi(y)\,\de y&=\lim_{k\to\infty}\int_{\R^n} u_k(y)\div^{s_k}\Psi(y)\,\de y\\
&=-\lim_{k\to\infty}\int_\Omega \Psi(y)\cdot\de D^{s_k} u_k(y)\\
&=-\lim_{k\to\infty}\int_\Omega \Psi(y)\cdot\de D v_k(y)\\
&=-\int_\Omega\Psi(y)\cdot \de Dv(y)=\int_\Omega v(y)\div\Psi(y)\,\de y.
\end{align*}
Therefore, there exists a constant $a\in\R$ such that 
\[
u=v+a\quad\text{in $L^1(\Omega)$}.
\]
Let us define
\[
w_k\coloneqq v_k-a_k+a\in BV(\Omega),\quad k\in\N.
\]
Then,
\begin{align*}
&u_k-w_k\to u-v-a=0& &\text{strongly in $L^1(\Omega)$ as $k\to\infty$},\\
&Dw_k=Dv_k=D^{s_k} u_k& &\text{in $\mathcal M_b(\Omega;\R^n)$ for all $k\in\N$},
\end{align*}
as required.
\end{proof}

We point out that, if $u_k\in BV^{s_k}(\Omega;T)$, and we apply Proposition~\ref{prop:rep}, then the new functions $w_k$ do not necessarily belong to $BV(\Omega;T)$ since $w_k$ does not take values in $T$. To solve this issue, we use the following anisotropic version of the coarea formula, which can be found in~\cite[Theorem~3]{Grasmair10}.

\begin{proposition}\label{prop:an-coarea}
Let $\Omega\subset\R^n$ be a bounded open set with Lipschitz boundary, let $\psi\in\mathcal G(\lambda,\Lambda)$, and $u\in BV(\Omega)$. Then,
\[
\int_\Omega \psi\left(y,\frac{\de Du}{\de |Du|}(y)\right)\de |Du|(y)=\int_{-\infty}^\infty\int_{\partial^*\{u<t\}\cap \Omega}\psi(y,\nu_{\{u<t\}}(y))\,\de \mathcal H^{n-1}(y)\,\de t.
\]
\end{proposition}

As a consequence of Proposition~\ref{prop:an-coarea}, we obtain the following result, which allows us to pass from a sequence $(w_k)_k\subset BV(\Omega)$ to a sequence $(z_k)_k\in BV(\Omega;T)$ without substantially increasing the energy $\mathcal E_k$ defined in~\eqref{eq:Ek}.

\begin{proposition}\label{prop:app-BVT}
Let $\Omega\subset\R^n$ be a bounded open set with Lipschitz boundary and $(\psi_k)_k\subset\mathcal G(\lambda,\Lambda)$. Assume that $(u_k)_k\subset BV(\Omega)$ and $u\in BV(\Omega;T)$ satisfy
\[
u_k\to u\quad\text{strongly in $L^1(\Omega)$ as $k\to\infty$}.
\]
Then, for all $\varepsilon\in (0,1)$ there exists a sequence $(v_k^\varepsilon)_k\in BV(\Omega;T)$ such that
\[
(1-\varepsilon) \int_\Omega \psi_k\left(y,\frac{\de Dv_k^\varepsilon}{\de |Dv_k^\varepsilon|}(y)\right)\de|Dv_k^\varepsilon|(y)\le\int_\Omega \psi_k\left(x,\frac{\de Du_k}{\de |Du_k|}(y)\right)\de|Du_k|(y)\quad\text{for all $k\in\N$},
\]
and satisfying
\begin{equation}\label{eq:vkdelta}
v_k^\varepsilon\to u\quad\text{strongly in $L^1(\Omega)$ as $k\to\infty$}.
\end{equation} 
\end{proposition}

\begin{proof}
We define
\[
w_k\coloneqq \min\{\max\{u_k, c_1 \}, c_M \}\quad\text{in $\Omega$ for all $k\in\N$}.
\]
Then, $(w_k)_k\subset BV(\Omega)$ and 
\begin{equation}\label{eq:wk-uk}
\int_\Omega\psi_k\left(y,\frac{\de Dw_k}{\de |Dw_k|}(y)\right)\de|Dw_k|(y)\le\int_\Omega \psi_k\left(y,\frac{\de Du_k}{\de |Du_k|}(y)\right)\de|Du_k|(y)\quad\text{for all $k\in\N$}.
\end{equation}
Moreover, since $u\in BV(\Omega;T)$, we have
\begin{equation}\label{eq:wk}
w_k\to \min\{\max\{u, c_1 \}, c_M \}=u\quad\text{strongly in $L^1(\Omega)$ as $k\to\infty$}.
\end{equation}
By using the anisotropic coarea formula of Proposition~\ref{prop:an-coarea}, we have
\begin{align*}
\int_\Omega \psi_k\left(y,\frac{\de Dw_k}{\de |Dw_k|}(y)\right)\de|Dw_k|(y)&= \int_{c_1}^{c_M} \int_{\partial^*\{w_k<t\}\cap \Omega}\psi_k(y,\nu_{\{w_k<t\}}(y))\,\de\mathcal H^{n-1}(y)\,\de t\\
&= \sum_{i=2}^M\int_{c_{i-1}}^{c_i} \int_{\partial^*\{w_k<t\}\cap \Omega}\psi_k(y,\nu_{\{w_k<t\}}(y))\,\de\mathcal H^{n-1}(y)\,\de t.
\end{align*}

Let us fix $\varepsilon\in (0,1)$. We set
\[
\theta\coloneqq \min_{i\in\{2,\dots,M\}}(c_i-c_{i-1}).
\]
 We claim that for all $i\in\{2,\dots,M\}$ there exist $t_k^i\in [c_{i-1}+\frac{\varepsilon\theta}{2},c_i-\frac{\varepsilon\theta}{2}] $ such that $\{w_k<t_k^i\}$ has finite perimeter in $\Omega$ and
\begin{align}\label{eq:tidelta}
& (c_i-c_{i-1}-\varepsilon\theta) \int_{\partial^*\{w_k<t_k^i\}\cap \Omega}\psi_k(y,\nu_{\{w_k<t_k^i\}}(y))\,\de\mathcal H^{n-1}(y)\nonumber\\
&\le \int_{c_{i-1}}^{c_i} \int_{\partial^*\{w_k<t\}\cap \Omega}\psi_k(y,\nu_{\{w_k<t\}}(y))\,\de\mathcal H^{n-1}(y)\,\de t.
\end{align}
Indeed, if~\eqref{eq:tidelta} is false for a.e.\ $t\in [c_{i-1}+\frac{\varepsilon\theta}{2},c_i-\frac{\varepsilon\theta}{2}]$, then we get
\begin{align*}
& (c_i-c_{i-1}-\varepsilon\theta) \int_{\partial^*\{w_k<t\}\cap \Omega}\psi_k(y,\nu_{\{w_k<t\}}(y))\,\de\mathcal H^{n-1}(y)\\
&> \int_{c_{i-1}}^{c_i} \int_{\partial^*\{w_k<t\}\cap \Omega}\psi_k(y,\nu_{\{w_k<t\}}(y))\,\de\mathcal H^{n-1}(y)\,\de t,
\end{align*}
which leads to a contradiction by integrating on $[c_{i-1}+\frac{\varepsilon\theta}{2},c_i-\frac{\varepsilon\theta}{2}]$. 
Thus,
\begin{align}\label{eq:nu-wk}
&(1-\varepsilon) \sum_{i=2}^M(c_i-c_{i-1}) \int_{\partial^*\{w_k<t_k^i\}\cap \Omega} \psi_k(y,\nu_{\{w_k<t_k^i\}}(y))\,\de\mathcal H^{n-1}(y)\nonumber\\
& \le \sum_{i=2}^M(c_i-c_{i-1}-\varepsilon\theta) \int_{\partial^*\{w_k<t_k^i\}\cap \Omega} \psi_k(y,\nu_{\{w_k<t_k^i\}}(y))\,\de\mathcal H^{n-1}(y) \nonumber\\
&\le \int_\Omega \psi_k\left(y,\frac{\de Dw_k}{\de |Dw_k|}(y)\right)\de|Dw_k|(y).
\end{align}
For all $k\in\N$ we define the sets
\[
E_k^1\coloneqq \{w_k<t_k^2\},\quad E_k^i\coloneqq\{t_k^i\le w_k<t_k^{i+1}\}\text{ for $i\in\{2,\dots,M-1\}$},\quad
E_k^M\coloneqq\{w_k\ge t_k^M\},
\]
and the functions $v_k^\varepsilon\colon \Omega\to T$ as 
\[
v_k^\varepsilon\coloneq\sum_{i=1}^M c_i \mathbf{1}_{E_k^i}.
\]
By construction $(v_k^\varepsilon)_k\subset BV(\Omega;T)$ and 
\begin{equation}\label{eq:vk-wk}
\{v_k^\varepsilon<t\}=\{w_k<t_k^i\}\quad\text{for all $t\in (c_{i-1},c_i) $ and $i\in \{2,\dots,M\}$}.
\end{equation}
In particular, by applying again Proposition~\ref{prop:an-coarea} and using~\eqref{eq:wk-uk} and~\eqref{eq:nu-wk}, we get
\begin{align*}
&(1-\varepsilon) \int_\Omega \psi_k\left(y,\frac{\de Dv_k^\varepsilon}{\de |Dv_k^\varepsilon|}(y)\right)\de|Dv_k^\varepsilon|(y)\\
&=(1-\varepsilon) \int_{c_1}^{c_M} \int_{\partial^*\{v_k^\varepsilon<t\}\cap\Omega} \psi_k\left(y,\nu_{\{v_k^\varepsilon<t\}}(y)\right)\de\mathcal H^{n-1}(y)\,\de t\\
&=(1-\varepsilon) \sum_{i=2}^M(c_i-c_{i-1}) \int_{{\partial^*\{w_k<t_k^i\}\cap \Omega}}\psi_k\left(y,\nu_{\{w_k<t_k^i\}}(y)\right)\de\mathcal H^{n-1}(y)\\
&\le\int_\Omega \psi_k\left(y,\frac{\de Du_k}{\de |Du_k|}(y)\right)\de|Du_k|(y).
\end{align*}

To show~\eqref{eq:vkdelta}, we first observe that~\eqref{eq:wk} implies the existence of a not relabeled subsequence such that, for all $i\in\{2,\dots,M\}$ and for a.e.\ $t\in(c_{i-1},c_i)$, we have
\begin{equation}\label{eq:wkt}
\mathbf{1}_{\{w_k < t\}} \to \mathbf{1}_{\{u < t\}} = \mathbf{1}_{\{u < c_i \}} \quad \text{strongly in $L^1(\Omega)$ as $k \to \infty$}.
\end{equation}
In particular, for all $i\in\{2,\dots,M\}$ we can find $s^i\in (c_{i-1},c_{i-1}+\frac{\varepsilon\theta}{2}) $ and $r^i\in (c_i-\frac{\varepsilon\theta}{2},c_i) $ such that~\eqref{eq:wkt} holds. Then, 
\begin{align*}
&\int_\Omega |\mathbf{1}_{\{w_k<t_k^i\}}(x)-\mathbf{1}_{\{u< c_i \}}(x)|\,\de x\\
&=\mathcal L^n(\{w_k<t_k^i\}\triangle\{u< c_i \})\le\mathcal L^n(\{w_k< r^i\}\setminus\{u< c_i \})+\mathcal L^n(\{u< c_i \}\setminus \{w_k<s^i\})\\
&\le\int_\Omega |\mathbf{1}_{\{w_k<r^i\}}(x)-\mathbf{1}_{\{u< c_i \}}(x)|\,\de x+\int_\Omega |\mathbf{1}_{\{w_k<s^i\}}(x)-\mathbf{1}_{\{u< c_i \}}(x)|\,\de x,
\end{align*}
where we used that $s^i<t_k^i<r^i$. By~\eqref{eq:wkt} this yields
\[
\mathbf{1}_{\{w_k<t_k^i\}}\to \mathbf{1}_{\{u< c_i \}}\quad\text{strongly in $L^1(\Omega)$ as $k\to\infty$ for all $i\in\{2,\dots,M\}$}.
\]
Therefore, by Fubini's theorem and~\eqref{eq:vk-wk} we derive
\begin{align*}
\int_\Omega|v_k^\varepsilon(x)-u(x)|\,\de x&= \int_{c_1}^{c_M} \int_\Omega|\mathbf{1}_{\{v_k^\varepsilon<t\}}(x)- \mathbf{1}_{\{u<t\}}(x)|\,\de x\,\de t\\
&=\sum_{i=2}^M (c_i-c_{i-1}) \int_\Omega|\mathbf{1}_{\{w_k<t_k^i\}}(x)- \mathbf{1}_{\{u< c_i \}}(x)|\,\de x\to 0\quad\text{as $k\to\infty$}.
\end{align*}
By the Urysohn property, we conclude the proof. 
\end{proof}

We can finally prove the $\Gamma$-liminf inequality.

\begin{proposition}[$\Gamma$-liminf inequality]\label{prop:lower-bound}
Let $\Omega\subset\R^n$ be a bounded open set with Lipchitz boundary. Let $(\psi_k)_k\subset \mathcal G(\lambda,\Lambda)$ and $(s_k)_k\subset (0,1)$ be such that $s_k\to 1$ as $k\to\infty$. Let $(\mathcal F_k)_k$ be defined as in~\eqref{eq:Fk}. Assume~\eqref{eq:psi-local}. Let $(u_k)_k\subset L^1(\Omega)$ and $u\in BV(\Omega;T)$ be such that 
\begin{equation}\label{eq:uku}
u_k\to u\quad\text{strongly in $L^1(\Omega)$ as $k\to\infty$}.
\end{equation}
Then,
\[
\liminf_{k\to\infty}\mathcal F_k(u_k)\ge\mathcal F_0(u). 
\]
\end{proposition}

\begin{proof}
Let $(u_k)_k\subset L^1(\Omega)$ and $u\in BV(\Omega;T)$ satisfy~\eqref{eq:uku}. Without loss of generality, we may assume that 
\[
\liminf_{k\to\infty}\mathcal F_k(u_k)=\lim_{k\to\infty}\mathcal F_k(u_k),\qquad \sup_{k\in \N}\mathcal F_k(u_k)<\infty.
\]
In particular, for all $k\in\N$ there exists $v_k\in \mathcal A^{s_k}(u_k,\Omega)$ such that 
\begin{equation}\label{eq:liminf1}
\int_\Omega \psi_k\left(y,\frac{\de D^{s_k}v_k}{|\de D^{s_k}v_k|}(y)\right)\de |D^{s_k}v_k|(y)\le \mathcal F_k(u_k)+\frac{1}{k},
\end{equation}
which yields
\[
\|v_k\|_{L^\infty(\R^n)}+|D^{s_k}v_k|(\Omega)\le \max_{i\in\{1,\dots,M\}}|c_i| +\frac{1}{\lambda}\left(\sup_{k\in\N}\mathcal F_k(u_k)+1\right)<\infty.
\]
By Lemma~\ref{lem:comp2}, we can find a not relabeled subsequence and $(w_k)_k\subset BV(\Omega)$ such that 
\begin{align*}
w_k\to u\quad\text{strongly in $L^1(\Omega)$ as $k\to\infty$},\qquad
Dw_k=D^{s_k}v_k\quad\text{in $\mathcal M_b(\Omega;\R^n)$ for all $k\in\N$}.
\end{align*}
This gives
\begin{equation}\label{eq:liminf2}
\int_\Omega \psi_k\left(y,\frac{\de D^{s_k}v_k}{\de |D^{s_k}v_k|}(y)\right)\de |D^{s_k}v_k|(y)= \mathcal E_k(w_k,\Omega)\quad\text{for all $k\in\N$},
\end{equation}
where $\mathcal E_k$ is the local functional defined in~\eqref{eq:Ek}.

Let $\varepsilon\in (0,1)$ be fixed. By Proposition~\ref{prop:app-BVT} applied to $(w_k)_k\subset BV(\Omega)$ and $u\in BV(\Omega;T)$, we can find $(z_k^\varepsilon)_k\subset BV(\Omega;T)$ such that 
\[
(1-\varepsilon) \mathcal E_k(z_k^\varepsilon,\Omega)\le \mathcal E_k(w_k,\Omega),\qquad z_k^\varepsilon\to u\quad\text{strongly in $L^1(\Omega)$ as $k\to\infty$}.
\]
Thus, thanks to Theorem~\ref{thm:local-gamma} and~\eqref{eq:liminf1}--\eqref{eq:liminf2} we get
\[
(1-\varepsilon) \mathcal F_0(u)\le (1-\varepsilon) \liminf_{k\to\infty}\mathcal E_k(z_k^\varepsilon,\Omega)\le \liminf_{k\to\infty}\mathcal E_k(w_k,\Omega)\le\lim_{k\to\infty}\mathcal F_k(u_k).
\]
By sending $\varepsilon\to 0^+$ we conclude.
\end{proof} 

\subsection{\texorpdfstring{$\Gamma$}{Gamma}-limsup inequality}

We now focus on the proof of the $\Gamma$-limsup inequality. The strategy is to exploit the uniform approximation condition~\eqref{A} to reduce the problem to the case where the densities have the form $\psi_k(x,\xi)=\sum_{i=1}^Nb^i_k(x)\varphi^i_k(\xi)$. We then use the Riesz potential to pass from the fractional variation to the classical total variation. This last step relies on the following convergence result.

\begin{lemma}\label{lem:a-app}
Let $\alpha\in (0,n)$, $R>1$, and $\varphi\in C_c^1(B_R)$. Let $b\colon\R^n\to\R$ be a uniformly continuous and bounded function, and let $r_b$ be the radius of uniform continuity of $b$ defined in~\eqref{eq:rb}. Then, for all $\eta>0$ we have
\begin{align}\label{eq:aphi-est}
\|\mathcal I^{\alpha}(b\phi)-b\phi\|_{L^\infty(\R^n)}&\le\max\left\{(3R)^\alpha-r_b(\eta)^\alpha,\frac{\alpha R^{\alpha}}{n}\right\}\frac{\omega_n}{\alpha\gamma_{\alpha}}\|b\|_{L^\infty(\R^n)}\|\phi\|_{L^\infty(B_R)}\nonumber\\
&\quad+\left|1-\frac{\omega_n r_b(\eta)^\alpha}{\alpha\gamma_{\alpha}}\right|\|b\|_{L^\infty(\R^n)}\|\phi\|_{L^\infty(B_R)}+\frac{\omega_n r_b(\eta)^\alpha \eta}{\alpha\gamma_{\alpha}}\|\phi\|_{L^\infty(B_R)}\nonumber\\
&\quad+\frac{\omega_n r_b(\eta)^{\alpha+1}}{(\alpha+1)\gamma_{\alpha}}\|b\|_{L^\infty(\R^n)}\|\nabla \phi\|_{L^\infty(B_R)}.
\end{align}
In particular, if $(\alpha_k)_k\subset (0,n)$ satisfies $\alpha_k\to 0$ as $k\to\infty$, and $(b_k)_k$ is a sequence of uniformly continuous and bounded functions satisfying 
\begin{equation}\label{eq:bk-conv}
\sup_{k\in\N}\|b_k\|_{L^\infty(\R^n)}<\infty,\qquad \alpha_k\log(r_{b_k}(\eta))\to 0\quad\text{as $k\to\infty$ for all $\eta>0$},
\end{equation}
then
\begin{equation}\label{eq:aphi-lim}
\lim_{k\to\infty}\|\mathcal I^{\alpha_k}(b_k\phi)-b_k\phi\|_{L^\infty(\R^n)}=0.
\end{equation}
\end{lemma}

\begin{proof}
Let $\eta>0$ be fixed. For all $x\in\R^n$ we define
\begin{align*}
b_\eta^1(x)\coloneqq\frac{1}{\gamma_{\alpha}}\int_{B_{r_b(\eta)}(x)}\frac{b(y)\phi(y)}{|x-y|^{n-\alpha}}\,\de y,\qquad b_\eta^2(x)\coloneqq\frac{1}{\gamma_{\alpha}}\int_{\R^n\setminus B_{r_b(\eta)}(x)}\frac{b(y)\phi(y)}{|x-y|^{n-\alpha}}\,\de y.
\end{align*}
For all $x\in B_{2R}$ we have $B_R\subset B_{3R}(x)$, which implies 
\begin{align*}
|b_\eta^2(x)|&\le \frac{1}{\gamma_{\alpha}}\int_{B_{3R}(x)\setminus B_{r_b(\eta)}(x)}\frac{|b(y)||\phi(y)|}{|x-y|^{n-\alpha}}\,\de y\le \left((3R)^\alpha-r_b(\eta)^\alpha\right)\frac{\omega_n}{\alpha\gamma_{\alpha}}\|b\|_{L^\infty(\R^n)}\|\phi\|_{L^\infty(B_R)}.
\end{align*}
Moreover, for all $x\in\R^n\setminus B_{2R}$ and $y\in B_R$ we get
\[
|x-y|\ge |x|-|y|\ge R.
\]
Thus, for $x\in\R^n\setminus B_{2R}$
\[
|b_\eta^2(x)|\le \frac{\omega_n R^{\alpha}}{n\gamma_{\alpha}}\|b\|_{L^\infty(\R^n)}\|\phi\|_{L^\infty(B_R)}.
\]
Hence,
\begin{equation}\label{eq:aphi-est1}
\|b_\eta^2\|_{L^\infty(\R^n)}\le \max\left\{(3R)^\alpha-r_b(\eta)^\alpha,\frac{\alpha R^{\alpha}}{n}\right\}\frac{\omega_n}{\alpha\gamma_{\alpha}}\|b\|_{L^\infty(\R^n)}\|\phi\|_{L^\infty(B_R)}.
\end{equation}
Moreover, for all $x\in\R^n$ we can write
\[
b(x)\phi(x)=\frac{1}{\gamma_{\alpha}}\int_{B_{r_b(\eta)}(x)}\frac{b(x)\phi(x)}{|x-y|^{n-\alpha}}\,\de y+b(x)\phi(x)\left(1-\frac{\omega_n r_b(\eta)^\alpha}{\alpha\gamma_{\alpha}}\right).
\]
Thus, by using~\eqref{eq:rb} we have
\begin{align}\label{eq:aphi-est2}
|b_\eta^1(x)-b(x)\phi(x)| 
&\le \left|1-\frac{\omega_n r_b(\eta)^\alpha}{\alpha\gamma_{\alpha}}\right|\|b\|_{L^\infty(\R^n)}\|\phi\|_{L^\infty(B_R)}+\frac{1}{\gamma_{\alpha}}\int_{B_{r_b(\eta)}(x)}\frac{|b(y)\phi(y)-b(x)\phi(x)|}{|x-y|^{n-\alpha}}\,\de y\nonumber\\
&\le \left|1-\frac{\omega_n r_b(\eta)^\alpha}{\alpha\gamma_{\alpha}}\right|\|b\|_{L^\infty(\R^n)}\|\phi\|_{L^\infty(B_R)}+\frac{\omega_n r_b(\eta)^\alpha \eta}{\alpha\gamma_{\alpha}}\|\phi\|_{L^\infty(B_R)}\nonumber\\
&\quad+\frac{\omega_nr_b(\eta)^{\alpha+1}}{(\alpha+1)\gamma_{\alpha}}\|b\|_{L^\infty(\R^n)}\|\nabla \phi\|_{L^\infty(B_R)}.
\end{align}
By combining~\eqref{eq:aphi-est1} and~\eqref{eq:aphi-est2} we get~\eqref{eq:aphi-est}. 

Finally, for all $\eta>0$, by~\eqref{eq:gammaa},~\eqref{eq:aphi-est}, and~\eqref{eq:bk-conv} we get
\[
\limsup_{k\to\infty}\|\mathcal I^{\alpha_k}(b_k\phi)-b_k\phi\|_{L^\infty(\R^n)}\le \eta\|\phi\|_{L^\infty(B_R)},
\]
which gives~\eqref{eq:aphi-lim}. 
\end{proof} 

We can finally prove the $\Gamma$-limsup inequality. 

\begin{proposition}[$\Gamma$-limsup inequality]\label{prop:upper-bound}
Let $\Omega\subset\R^n$ be a bounded open set with Lipschitz boundary. Let $(\psi_k)_k\subset\mathcal G(\lambda,\Lambda)$ satisfy the uniform approximation condition~\eqref{A}. Let $(s_k)_k \subset (0,1)$ satisfy $s_k\to 1$ as $k\to\infty$ and the compatibility condition~\eqref{eq:compatibility}. Assume~\eqref{eq:psi-local}. Then, for all functions $u\in BV(\Omega;T)$ there exists a sequence $(w_k)_k\subset L^1(\Omega)$ satisfying
\[
w_k\to u\quad\text{strongly in $L^1(\Omega)$ as $k\to\infty$},\qquad \limsup_{k\to\infty}\mathcal F_k(w_k)\le \mathcal F_0 (u). 
\]
\end{proposition}

\begin{proof}
Let $u\in BV(\Omega;T)$ be fixed. By Proposition~\ref{prop:extension} there exists a function $v\in BV(\R^n;T)$ such that 
\[
v=u\quad\text{a.e.\ in $\Omega$},\qquad |Dv|(\partial\Omega)=0.
\]
For $\varepsilon>0$ let $\Omega^\varepsilon\subset\R^n$ be a bounded open set with Lipschitz boundary satisfying
\begin{equation}\label{eq:eta}
\Omega\subset\subset\Omega^\varepsilon,\qquad |Dv|(\Omega^\varepsilon\setminus\Omega)<\varepsilon.
\end{equation}
By Theorem~\ref{thm:local-gamma}, there exists $(v_k^\varepsilon)_k\subset BV(\Omega^\varepsilon;T)$ such that 
\begin{equation}\label{eq:rec}
v_k^\varepsilon\to v\quad\text{strongly in $L^1(\Omega^\varepsilon)$ as $k\to\infty$},\qquad\limsup_{k\to\infty}\mathcal E_k(v_k^\varepsilon,\Omega^\varepsilon)\le\mathcal F_0 (v,\Omega^\varepsilon).
\end{equation}
In particular, since $(\psi_k)_k\subset\mathcal G(\lambda,\Lambda)$, the sequence $(v_k^\varepsilon)_k\subset BV(\Omega^\varepsilon;T)$ satisfies
\[
\sup_{k\in\N}|Dv_k^\varepsilon|(\Omega^\varepsilon)<\infty.
\]
For all $k\in\N$ we define
\begin{equation}\label{eq:wke}
w_k^\varepsilon(x) \coloneqq 
\begin{cases}
v_k^\varepsilon(x)&\quad\text{if $x\in\Omega^\varepsilon$},\\
 c_1 &\quad\text{if $x\in\R^n\setminus\Omega^\varepsilon$}.
\end{cases}
\end{equation}
Then, we have that $(w_k^\varepsilon)_k\subset BV(\R^n;T)$ and
\begin{align}\label{eq:Dww}
\sup_{k\in\N}|Dw_k^\varepsilon|(\R^n)\le \sup_{k\in\N}|Dv_k^\varepsilon|(\Omega^\varepsilon)+ (c_M-c_1) \mathcal H^{n-1}(\partial\Omega^\varepsilon)<\infty.
\end{align}
By Propositions~\ref{prop:I1-sDu} and~\ref{prop:Ds-nablas}, we deduce that $w_k^\varepsilon\in BV^{s_k}(\Omega;T)$ for all $k\in\N$ and 
\begin{equation}\label{eq:Dskwkeps}
D^{s_k}w_k^\varepsilon=\nabla^{s_k}w_k^\varepsilon\,\mathcal L^n=\mathcal I^{1-s_k}Dw_k^\varepsilon\,\mathcal L^n\quad\text{in $\mathcal M_b(\Omega;\R^n)$}.
\end{equation}
Moreover, in view of~\eqref{eq:V1s} and Remark~\ref{rem:behaviour}, we can find a constant $C>0$, independent of $k\in\N$, such that 
\begin{equation}\label{eq:Dskwk-est}
|D^{s_k}w_k^\varepsilon|(\Omega)\le C(1+|Dw_k^\varepsilon|(\R^n))\quad\text{for all $k\in\N$}.
\end{equation}

Let $\delta>0$ be fixed and let $(b_k^i)_{i=1}^N,(\varphi_k^i)_{i=1}^N\in C(\R^n)$ be the families of functions given by assumption~\eqref{A}. By exploiting~\eqref{eq:Dskwkeps} we have 
\begin{equation}\label{eq:varphiDswk}
\varphi_k^i\left(\frac{\de D^{s_k}w_k^\varepsilon}{ \de |D^{s_k}w_k^\varepsilon|}(y)\right) |D^{s_k}w_k^\varepsilon|(y)=\varphi_k^i(D^{s_k}w_k^\varepsilon)=\varphi_k^i(\nabla^{s_k}w_k^\varepsilon)\,\mathcal L^n\quad\text{in $\mathcal M_b(\Omega;\R^n)$}.
\end{equation}
Moreover, since $Dw_k^\varepsilon\in\mathcal M_b(\R^n;\R^n)$ all $k\in\N$, we derive that $\varphi_k^i(Dw_k^\varepsilon)\in \mathcal M_b(\R^n)$ for all $i\in\{1,\dots,N\}$. In particular, by~\eqref{eq:wke}
\begin{equation}\label{eq:varphiwk}
\varphi_k^i(Dw_k^\varepsilon)=\varphi_k^i(Dv_k^\varepsilon)=\varphi_k^i\left(\frac{\de Dv_k^\varepsilon}{\de |Dv_k^\varepsilon|}\right) |Dv_k^\varepsilon|\quad\text{in $\mathcal M_b(\Omega^\varepsilon;\R^n)$}.
\end{equation}
Thus, for all $k\in\N$ we get (see Proposition~\ref{prop:mizuta})
\begin{equation}\label{eq:Dskwkeps2}
\mathcal I^{1-s_k}\varphi_k^i(Dw_k^\varepsilon)\in L^1_{\rm loc}(\R^n)\quad\text{for all $i\in\{1,\dots,N\}$},
\end{equation}
and by~\eqref{eq:Dskwkeps} and Jensen's inequality for sublinear functions of measures (see~\cite[Theorem~1]{Goffman-Serrin})
\begin{equation}\label{eq:varphii}
\varphi_k^i(\nabla^{s_k}w_k^\varepsilon)=\varphi_k^i(\mathcal I^{1-s_k}Dw_k^\varepsilon)\le \mathcal I^{1-s_k}\varphi_k^i(Dw_k^\varepsilon)\quad\text{a.e.\ in $\R^n$ for all $i\in\{1,\dots,N\}$}.
\end{equation}

Let us fix a function $\phi\in C_c^1(\Omega^\varepsilon)$ such that $0\le \phi\le 1$ in $\Omega^\varepsilon$ and $\phi=1$ on $\overline\Omega$. By using~\eqref{A},~\eqref{eq:Dskwk-est},~\eqref{eq:varphiDswk}, and~\eqref{eq:varphii}, for all $k\in\N$ we get
\begin{align}
&\int_\Omega \psi_k\left(y,\frac{\de D^{s_k}w_k^\varepsilon}{\de |D^{s_k}w_k^\varepsilon|}(y)\right)\de |D^{s_k}w_k^\varepsilon|(y)\nonumber\\
&\le \sum_{i=1}^N\int_\Omega b_k^i(y)\varphi_k^i\left(\frac{\de D^{s_k}w_k^\varepsilon}{\de |D^{s_k}w_k^\varepsilon|}(y)\right)\de |D^{s_k}w_k^\varepsilon|(y)+\delta\sup_{k\in\N}|D^{s_k}w_k^\varepsilon|(\Omega)\nonumber\\
&\le \sum_{i=1}^N\int_\Omega b_k^i(y)\varphi_k^i(\nabla^{s_k}w_k^\varepsilon(y))\,\de y+C \delta(1+\sup_{k\in\N}|Dw_k^\varepsilon|(\R^n))\nonumber\\
&\le \sum_{i=1}^N\int_{\R^n}b_k^i(y)\phi(y)\mathcal I^{1-s_k}\varphi_k^i(Dw_k^\varepsilon)(y)\,\de y+C \delta(1+\sup_{k\in\N}|Dw_k^\varepsilon|(\R^n)).\label{eq:limsup1}
\end{align}
Thanks to Fubini's theorem, recalling that $\varphi^i_k(Dw_k^\varepsilon)\in\mathcal M_b(\R^n)$, the identity
\begin{equation}\label{eq:limsup2}
\sum_{i=1}^N\int_{\R^n}b_k^i(y)\phi(y)\mathcal I^{1-s_k}\varphi_k^i(Dw_k^\varepsilon)(y)\,\de y=\sum_{i=1}^N\int_{\R^n} \mathcal I^{1-s_k}(b_k^i \phi)(y)\,\de \varphi_k^i(Dw_k^\varepsilon)(y)
\end{equation}
holds for all $k\in\N$. Moreover, by using again~\eqref{A},~\eqref{eq:varphiwk}, and the fact that $\phi \in C^1_c(\Omega^\varepsilon)$ with $0\le\phi\le 1$ in $\Omega^\varepsilon$, we get
\begin{align}
&\sum_{i=1}^N\int_{\R^n} \mathcal I^{1-s_k}(b_k^i \phi)(y)\,\de \varphi_k^i(Dw_k^\varepsilon)(y)\nonumber\\
&\le \sum_{i=1}^N\int_{\Omega^\varepsilon}b_k^i(y)\phi(y)\de \varphi_k^i(Dw_k^\varepsilon)(y) +\sup_{k\in\N}|Dw_k^\varepsilon|(\R^n)\sum_{i=1}^N\|\varphi_k^i
\|_{C(\mathbb S^{n-1})}\|\mathcal I^{1-s_k}(b_k^i \phi)-b_k^i \phi\|_{L^\infty(\R^n)}\nonumber\\
&\le \sum_{i=1}^N\int_{\Omega^\varepsilon}b_k^i(y)\varphi_k^i\left(\frac{\de Dv_k^\varepsilon}{\de |Dv_k^\varepsilon|}\right)\de |Dv_k^\varepsilon|(y) +\sup_{k\in\N}|Dw_k^\varepsilon|(\R^n)\sum_{i=1}^N\|\varphi_k^i
\|_{C(\mathbb S^{n-1})}\|\mathcal I^{1-s_k}(b_k^i \phi)-b_k^i \phi\|_{L^\infty(\R^n)}\nonumber\\
&\le \mathcal E_k(v_k^\varepsilon,\Omega^\varepsilon)+\delta\sup_{k\in\N}|Dv_k^\varepsilon|(\Omega^\varepsilon)+\sup_{k\in\N}|Dw_k^\varepsilon|(\R^n)\sum_{i=1}^N\|\varphi_k^i
\|_{C(\mathbb S^{n-1})}\|\mathcal I^{1-s_k}(b_k^i \phi)-b_k^i \phi\|_{L^\infty(\R^n)}.\label{eq:limsup3}
 \end{align}
Thanks to~\eqref{A} and~\eqref{eq:compatibility}, we can apply Lemma~\ref{lem:a-app} to deduce that
\begin{equation}\label{eq:limsup4}
\lim_{k\to\infty}\left(\sum_{i=1}^N\|\varphi_k^i
\|_{C(\mathbb S^{n-1})}\|\mathcal I^{1-s_k}(b_k^i \phi)-b_k^i \phi\|_{L^\infty(\R^n)}\right)=0.
\end{equation}
Therefore, by combining~\eqref{eq:rec} with~\eqref{eq:Dww} and~\eqref{eq:limsup1}--\eqref{eq:limsup4}, for all $\varepsilon,\delta>0$ we have
\begin{align*}
&\limsup_{k\to\infty}\int_\Omega \psi_k\left(y,\frac{\de D^{s_k}w_k^\varepsilon}{\de |D^{s_k}w_k^\varepsilon|}(y)\right)\de |D^{s_k}w_k^\varepsilon|(y)\\
&\le \mathcal F_0 (v,\Omega^\varepsilon)+C \delta(1+\sup_{k\in\N}|Dw_k^\varepsilon|(\R^n))+\delta\sup_{k\in\N}|Dv_k^\varepsilon|(\Omega^\varepsilon).
\end{align*} 
By sending $\delta\to 0$, for all $\varepsilon>0$ we obtain by~\eqref{eq:Dww}
\begin{align*}
\limsup_{k\to\infty}\mathcal F_k(w_k^\varepsilon)&\le\limsup_{k\to\infty}\int_\Omega \psi_k\left(y,\frac{\de D^{s_k}w_k^\varepsilon}{\de |D^{s_k}w_k^\varepsilon|}(y)\right)\de |D^{s_k}w_k^\varepsilon|(y)\\
&\le \mathcal F_0 (v,\Omega^\varepsilon)\le \mathcal F_0 (u)+\Lambda\varepsilon,
\end{align*} 
where the last step follows from~\eqref{eq:eta} and~\eqref{eq:psi0}. Sending $\varepsilon\to 0$ we can eventually conclude by a diagonal argument. 
\end{proof} 








\section*{Acknowledgments}

The work of S.\ Almi was funded by the FWF Austrian Science Fund through the Project 10.55776/P35359 and by the University of Naples Federico II through FRA Project ``ReSinApas”.

The work of S.\ Almi and F.\ Solombrino is part of the MUR - PRIN 2022, project Variational Analysis of Complex Systems in Materials Science, Physics and Biology, No.\ 2022HKBF5C, funded by European Union NextGenerationEU. 

The work of M.\ Caponi was founded by the European Union NextGenerationEU under the Italian Ministry of University and Research (MUR) National Centre for HPC, Big Data and Quantum Computing (CN\_00000013 – CUP: E13C22001000006). 

S.\ Almi, M.\ Caponi, and F.\ Solombrino are part of the Gruppo Nazionale per l'Analisi Matematica, la Probabilit\`a e le loro Applicazioni (INdAM-GNAMPA), and acknowledge the support of the INdAM-GNAMPA 2025 Project ``DISCOVERIES - Difetti e Interfacce in Sistemi Continui: un’Ottica Variazionale in Elasticità con Risultati Innovativi ed Efficaci Sviluppi'' (CUP: E5324001950001).


\appendix

\section{Auxiliary results}\label{app:a}

This appendix is devoted to the proofs of Proposition~\ref{prop:gl-est1}, the approximation result of Proposition~\ref{prop:app}, and Lemma~\ref{lem:BVO-comp}. We begin with the $L^p$-type estimates for the Riesz fractional gradient and divergence.

\begin{proof}[Proofs of Proposition~\ref{prop:gl-est1}]
We only prove~\eqref{eq:est-1} for $p\in [1,\infty)$, since the other estimates follow by a similar argument. 

Let $p\in [1,\infty)$ be fixed. We recall the following Minkowski's inequality
\begin{equation*}
\left(\int_{\R^n}\left(\int_{\R^n}|f(x,y)|\,\de y\right)^p \de x \right)^{\frac{1}{p}}
\le\int_{\R^n}\left(\int_{\R^n}|f(x,y)|^p\,\de x\right)^{\frac{1}{p}}\de y,
\end{equation*}
which holds for all measurable functions $f \colon\R^n \times\R^n\to\R $. Then, we have
\begin{align*}
\|\nabla^s\psi\|_{L^p(\R^n)}&\le\mu_s\int_{\R^n}\frac{\|\psi(\,\cdot\,+h)-\psi(\,\cdot\,)\|_{L^p(\R^n)}}{|h|^{n+s}}\,\de h\\
&\le \mu_s\left(\|\nabla\psi\|_{L^p(\R^n)}\int_{B_R}\frac{1}{|h|^{n+s-1}}\,\de h+2\|\psi\|_{L^p(\R^n)}\int_{\R^n\setminus B_R}\frac{1}{|h|^{n+s}}\,\de h\right)\\
&=\frac{\omega_nR^{1-s}\mu_s}{1-s}\|\nabla\psi\|_{L^p(\R^n)}+\frac{2\omega_n\mu_s}{sR^s}\|\psi\|_{L^p(\R^n)}.
\end{align*}
By minimizing on $R>0$ we derive~\eqref{eq:est-1}. 
\end{proof}

To prove Proposition~\ref{prop:app}, we begin by recalling some preliminary results. First, we show that any function $\psi \in L^\infty(\R^n) \cap C^1(\R^n)$ admits a weak $s$-fractional gradient, which can be expressed by formula~\eqref{eq:nablas}. The key observation is that formula~\eqref{eq:nablas} does not require $\psi$ to have compact support. Consequently, the integration by parts formula~\eqref{eq:dual} can also be extended to all functions $\psi\in L^\infty(\R^n) \cap C^1(\R^n)$ without the assumption of compact support.

\begin{lemma}\label{lem:weak-grad}
Let $s\in (0,1)$ and $\psi\in L^\infty(\R^n)\cap C^1(\R^n)$. Then, $\psi$ has a weak $s$-fractional gradient $v\in L^1_{\rm loc}(\R^n;\R^n)$ and
\begin{equation}\label{eq:v}
v(x)=\mu_s\int_{\R^n}\frac{(\psi(y)-\psi(x))(y-x)}{|y-x|^{n+s+1}}\,\de y\quad\text{for all $x\in\R^n$}.
\end{equation} 
\end{lemma}

\begin{proof}
We fix $s\in (0,1)$ and $\psi\in L^\infty(\R^n)\cap C^1(\R^n)$, and we consider the function $v$ defined in~\eqref{eq:v}
Notice that $v(x)$ is well-defined for all $x\in\R^n$ and $v\in L^1_{\rm loc}(\R^n;\R^n)$. Indeed, for all $R>0$ and $x\in B_R$ we have
\begin{align*}
|v(x)|&\le \mu_s\int_{B_R(x)}\frac{|\psi(x)-\psi(y)|}{|x-y|^{n+s}}\,\de y+\int_{\R^n\setminus B_R(x)}\frac{|\psi(x)-\psi(y)|}{|x-y|^{n+s}}\,\de y\\
&\le \frac{\omega_nR^{1-s}\mu_s}{1-s}\|\nabla \psi\|_{L^\infty (B_{2R})}+\frac{2\omega_n\mu_s}{sR^s}\|\psi\|_{L^\infty(\R^n)}.
\end{align*}

We claim that for all functions $\Psi\in C^1_c(\R^n;\R^n)$ we have
\[
\int_{\R^n}v(y)\cdot\Psi(y)\,\de y=-\int_{\R^n}\psi(y)\div^s\Psi(y)\,\de y.
\]
We follow the lines of the proof of~\cite[Lemma~2.5]{CS19}. For all $\Psi\in C^1_c(\R^n;\R^n)$ we have
\begin{align*}
\int_{\R^n}v(x)\cdot \Psi(x)\,\de x&=\mu_s\int_{\R^n}\int_{\R^n}\frac{(\psi(x)-\psi(y))(x-y)}{|x-y|^{n+s+1}}\,\de y\cdot\Psi(x)\,\de x\\
&=\mu_s\lim_{\varepsilon\to 0^+}\int_{\R^n}\int_{\R^n\setminus B_\varepsilon(x)}\frac{(\psi(x)-\psi(y))(x-y)}{|x-y|^{n+s+1}}\cdot\Psi(x)\,\de y\,\de x\\
&=-\mu_s\lim_{\varepsilon\to 0^+}\int_{\R^n}\int_{\R^n\setminus B_\varepsilon(y)}\frac{\psi(y)(x-y)\cdot(\Psi(x)-\Psi(y))}{|x-y|^{n+s+1}}\,\de x\,\de y\\
&=-\mu_s\int_{\R^n}\int_{\R^n}\frac{\psi(y)(x-y)\cdot(\Psi(x)-\Psi(y))}{|x-y|^{n+s+1}}\,\de x\,\de y=-\int_{\R^n}\psi(y)\div^s\Psi(y)\,\de y.
\end{align*}
Thus, $v\in L^1_{\rm loc}(\R^n;\R^n)$ is the weak $s$-fractional gradient of $\psi$ in $\R^n$. 
\end{proof}

In the following result, we extend the fractional Leibniz formula~\eqref{eq:leibniz} to all functions $\psi, \phi\in L^\infty(\R^n) \cap C^1(\R^n)$.

\begin{lemma}\label{lem:leibniz2}
Let $s\in (0,1)$. For all functions $\psi,\phi\in L^\infty(\R^n)\cap C^1(\R^n)$ the nonlocal operator $\nabla^s_{\rm NL}(\psi,\phi)\colon\R^n\to\R^n$ introduced in~\eqref{eq:nablaNL} is well-defined and
\begin{equation}\label{eq:leibniz3}
\nabla^s(\psi\phi)(x)=\nabla^s\psi(x) \phi(x)+\psi(x)\nabla^s\phi(x)+\nabla^s_{\rm NL}(\psi,\phi)(x)\quad\text{for all $x\in\R^n$}.
\end{equation}
\end{lemma}

\begin{proof}
We first observe that $\nabla^s_{\rm NL}(\psi, \phi)$ is well-defined for all $x\in\R^n$, by arguing as in the proof of Lemma~\ref{lem:weak-grad}. Then, the identity~\eqref{eq:leibniz3} follows from the definitions of $\nabla^s$ and $\nabla^s_{\rm NL}$.
\end{proof}

We are now in a position to prove Proposition~\ref{prop:app}.

\begin{proof}[Proof of Proposition~\ref{prop:app}]

The proof is divided in two steps. First, we approximate a function $u\in L^\infty(\R^n)$ satisfying
\[
|D^s u|(A) <\infty \quad \text{for all open sets $A \subset\subset \Omega$},
\]
by a sequence $(u_\varepsilon)_\varepsilon \subset L^\infty(\R^n)\cap C^\infty(\R^n)$ as $\varepsilon\to 0$. In the second step, we assume that $u\in L^\infty(\R^n)\cap C^\infty(\R^n)$ and approximate it by a sequence $(u_R)_R \subset C_c^\infty(\R^n)$ as $R\to\infty$. The conclusion then follows by a diagonal argument.

{\bf Step 1}. We fix a function $\rho\in C_c^\infty(B_1)$ which satisfies
\[
\rho(x)\ge 0\quad\text{for all $x\in\R^n$},\qquad \rho(x)=\rho(-x)\quad\text{for all $x\in\R^n$},\qquad\int_{B_1}\rho(y)\,\de y=1,
\]
and we set
\[
\rho_\varepsilon(x)\coloneqq \frac{1}{\varepsilon^n}\rho\left(\frac{x}{\varepsilon}\right)\quad\text{for $x\in\R^n$ and $\varepsilon>0$}.
\]

We define
\[
u_\varepsilon(x)\coloneqq (\rho_\varepsilon*u)(x)=\int_{B_\varepsilon(x)}u(y)\rho_\varepsilon(x-y)\,\de y\quad\text{for all $x\in\R^n$ and $\varepsilon>0$}.
\]
Clearly, $(u_\varepsilon)_\varepsilon\subset L^\infty(\R^n)\cap C^\infty(\R^n)$ and
\begin{align*}
&\|u_\varepsilon\|_{L^\infty(\R^n)}\le \|u\|_{L^\infty(\R^n)}& &\text{for all $\varepsilon>0$},\\
&u_\varepsilon\to u& &\text{strongly in $L^1_{\rm loc}(\R^n)$ for all $p\in [1,\infty)$ as $\varepsilon\to 0^+$}.
\end{align*}
In particular, by Lemma~\ref{lem:weak-grad} and Proposition~\ref{prop:Ds-nablas} there exists $D^s u_\varepsilon\in\mathcal M(\R^n;\R^n)$ and
\[
D^s u_\varepsilon=\nabla^s u_\varepsilon\,\mathcal L^n\quad\text{in $\mathcal M(\R^n;\R^n)$},
\]
where $\nabla^s u_\varepsilon$ is given by formula~\eqref{eq:nablas}. We fix a function $\Psi\in C^\infty_c(\Omega;\R^n)$. By Proposition~\ref{prop:Riesz}, for all $\varepsilon\in (0,\dist(\supp\Psi,\partial\Omega))$ we have
\begin{align*}
\int_{\R^n}\Psi(y)\cdot\de D^s u_\varepsilon(y)&=-\int_{\R^n}u_\varepsilon(y)\div^s\Psi(y)\,\de y\\
&=-\int_{\R^n}(\rho_\varepsilon*u)(y)\div^s\Psi(y)\,\de y\\
&=-\int_{\R^n}u(y)\div^s(\rho_\varepsilon*\Psi)(y)\,\de y\\
&=\int_{\R^n}(\rho_\varepsilon*\Psi)(y)\cdot \de D^s u(y)=\int_\Omega\Psi(y)\cdot \de (\rho_\varepsilon* D^s u)(y).
\end{align*}
This implies that 
\[
D^su_\varepsilon\rightharpoonup D^su\quad\text{weakly* in $\mathcal M(\Omega;\R^n)$ as $\varepsilon\to 0^+$}.
\]

{\bf Step 2.} Assume now that $u\in L^\infty(\R^n)\cap C^\infty(\R^n)$. We fix a family of functions $(\eta_R)_R\subset C_c^\infty(\R^n)$ which satisfies for all $R>0$ 
\[
0\le \eta_R\le 1\quad\text{in $\R^n$},\qquad \eta_R=1\quad\text{in $B_R$},\qquad \supp(\eta_R)\subset B_{R+1},\qquad |\nabla \eta_R|\le 2\quad\text{in $\R^n$}.
\]

We define
\[
u_R(x)\coloneqq u(x)\eta_R(x)\quad\text{for all $x\in\R^n$ and $R>0$}.
\]
We have that $u_R\in C^\infty_c(\R^n)$ and
\[
u_R\to u\quad\text{strongly in $L^1_{\rm loc}(\R^n)$ for all $p\in [1,\infty)$ as $R\to\infty$}.
\]
Moreover, by Lemma~\ref{lem:leibniz2} we have
\begin{align*}
\nabla^s u_R(x)=u(x)\nabla^s \eta_R(x)+\nabla^s u(x)\eta_R(x)+\nabla^s_{\rm NL}(u,\eta_R)(x)\quad\text{for all $x\in\R^n$ and $R>0$}.
\end{align*}

We fix $r>0$. For all $x\in B_r$ and $R>2r$ we get
\begin{align*}
&|u(x)\nabla^s\eta_R(x)|\le\mu_s\|u\|_{L^\infty(\R^n)}\int_{\R^n\setminus B_r(x)}\frac{|\eta_R(x)-\eta_R(y)|}{|x-y|^{n+s}}\,\de y\le \frac{\omega_n\mu_s}{sr^s}\|u\|_{L^\infty(\R^n)},\\
&|\nabla^s u(x)\eta_R(x)|\le |\nabla^s u(x)|\\
&|\nabla^s_{\rm NL}(u,\eta_R)(x)|\le 2\mu_s\|u\|_{L^\infty(\R^n)}\int_{\R^n\setminus B_r(x)}\frac{|\eta_R(x)-\eta_R(y)|}{|x-y|^{n+s}}\,\de y\le \frac{2\omega_n\mu_s}{sr^s}\|u\|_{L^\infty(\R^n)}.
\end{align*}
Thus,
\[
|D^s u_R|(B_r)=\int_{B_r}|\nabla^s u_R(y)|\,\de y\le \frac{3\omega_n^2r^{n-s}\mu_s}{n s}\|u\|_{L^\infty(\R^n)}+|D^s u|(B_r)\quad\text{for all $R>2r$}.
\]
Hence, there exist a not relabeled subsequence and a Radon measure $\mu\in \mathcal M(\R^n;\R^n)$ such that
\[
D^su_R\rightharpoonup\mu\quad\text{weakly* in $\mathcal M(\R^n;\R^n)$ as $R\to\infty$}.
\]
On the other hand, for all $\Psi\in C^\infty_c(\R^n;\R^n)$ we have
\begin{align*}
\int_{\R^n}\Psi(y)\cdot\de \mu(y)&=\lim_{R\to\infty}\int_{\R^n}\Psi(y)\cdot\de D^su_R(y)\\
&=-\lim_{R\to\infty}\int_{\R^n}u(y)\eta_R(y)\div^s\Psi(y)\,\de y=-\int_{\R^n}u(y)\div^s\Psi(y)\,\de y.
\end{align*}
Thus, $\mu=D^s u$, which implies $D^su_R\rightharpoonup D^su$ weakly* in $\mathcal M(\R^n;\R^n)$ as $R\to\infty$ by the Urysohn property.
\end{proof}

We conclude the appendix with the proof of Lemma~\ref{lem:BVO-comp}.

\begin{proof}[Proof of Lemma~\ref{lem:BVO-comp}]
(1) Let $(w_k)_k \subset BV(\Omega)$ be a sequence such that
\[
\sup_{k\in\N}|Dw_k|(\Omega)<\infty.
\]
By Poincare's inequality, there exists a constant $C>0$, independent of $k\in\N$, and a sequence $(a_k)_k \subset\R$ such that
\[
\|w_k-a_k\|_{L^1(\Omega)} \le C|Dw_k|(\Omega) \quad \text{for all $k\in\N$}.
\]
Hence, the sequence of functions
\[
v_k \coloneqq w_k-a_k\in BV(\Omega), \quad k\in\N,
\]
is uniformly bounded in $BV(\Omega)$. Therefore, there exist a not relabeled subsequence and a function $w\in BV(\Omega)$ such that~\eqref{eq:BVcon1}--\eqref{eq:BVcon2} hold.

(2) Let $(\Omega_j)_j$ be a sequence of open bounded sets with smooth boundaries such that
\[
\Omega_j\subset\subset\Omega_{j+1}\quad\text{for all $j\in\N$},\qquad \bigcup_{j\in\N}\Omega_j=\Omega.
\]
We fix $j\in\N$. By Poincare's inequality, there exists a constant $C_j>0$, independent of $k\in\N$, and a sequence $(a_{k,j})_k \subset\R$ such that
\begin{equation*}
\|w_k-a_{k,j}\|_{L^1(\Omega_j)}\le C_j|Dw_k|(\Omega_j)\quad\text{for all $k\in\N$}.
\end{equation*}
Let us set $a_k\coloneqq a_{k,1}$. Then, for all $j,k\in\N$ we have
\begin{align*}
\|w_k-a_k\|_{L^1(\Omega_j)}&\le \|w_k-a_{k,j}\|_{L^1(\Omega_j)}+\|a_{k,j}-a_k\|_{L^1(\Omega_j)}\\
&\le C_j|Dw_k|(\Omega_j)+\frac{\mathcal L^n(\Omega_j)}{\mathcal L^n(\Omega_1)}\|a_{k,j}-a_k\|_{L^1(\Omega_1)}\\
&\le \left(C_j+\frac{\mathcal L^n(\Omega_j)}{\mathcal L^n(\Omega_1)}C_1+\frac{\mathcal L^n(\Omega_j)}{\mathcal L^n(\Omega_1)}C_j\right)|Dw_k|(\Omega_j).
\end{align*} 
In particular, the sequence of functions
\[
v_k\coloneqq w_k-a_k\in BV_{\rm loc}(\Omega),\quad k\in\N,
\]
is uniformly bounded in $BV(\Omega_j)$ for all $j\in\N$. Therefore, for all $j\in\N$ there exist a subsequence $(k_{h,j})_h\subset\N$ and a function $v_j\in BV(\Omega_j)$ such that 
\begin{align*}
&v_{k_{h,j}}\to v_j & &\text{strongly in $L^1(\Omega_j)$ as $h\to\infty$},\\
&Dw_{k_{h,j}}=Dv_{k_{h,j}}\rightharpoonup Dv_j & &\text{weakly* in $\mathcal M_b(\Omega_j;\R^n)$ as $h\to\infty$}.
\end{align*} 
By a diagonal argument, we can find a subsequence $(k_h)_h\subset\N$ and a function $w\in BV_{\rm loc}(\Omega)$ such that~\eqref{eq:BVcon3} holds. 

\end{proof}



\begin{thebibliography}{99}

\bibitem{ACFS25}
{\sc S.\ Almi, M.\ Caponi, M.\ Friedrich, F.\ Solombrino},
{\em A fractional approach to strain-gradient plasticity: beyond core-radius of discrete dislocations},
Math.\ Ann.\ {\bf 391} (2025), 4063--4115

\bibitem{AB90-1}
{\sc L.\ Ambrosio, A.\ Braides},
{\em Functionals defined on partitions in sets of finite perimeter. I. Integral representation and $\Gamma$-convergence},
J.\ Math.\ Pures Appl.\ {\bf 69} (1990), 285--305

\bibitem{AB90-2}
{\sc L.\ Ambrosio, A.\ Braides},
{\em Functionals defined on partitions in sets of finite perimeter. II. Semicontinuity, relaxation and homogenization},
J.\ Math.\ Pures Appl.\ {\bf 69} (1990), 307--333

\bibitem{Ambrosio-Martinazzi}
{\sc L.\ Ambrosio, G.\ De Philippis, L.\ Martinazzi}, {\em Gamma-convergence of nonlocal perimeter functionals}, Manuscripta Math., {\bf 134} (2011), 377--403

\bibitem{AFP00}
{\sc L.\ Ambrosio, N.\ Fusco, D.\ Pallara},
{\em Functions of bounded variation and free discontinuity problems},
Oxford Math.\ Monogr. The Clarendon Press, Oxford University Press, New York, 2000. xviii+434 pp.

\bibitem{Antil-Bartels}
{\sc H.\ Antil, S.\ Bartels}, 
{\em Spectral approximation of fractional PDEs in image processing and phase field modeling}, 
Comput.\ Methods Appl.\ Math.\ {\bf 17} (2017), 661--678

\bibitem{Antiletal}
{\sc H.\ Antil, Z.W.\ Di, R.\ Khatri}, 
{\em Bilevel optimization, deep learning and fractional Laplacian regularization with applications in tomography}, 
Inverse Problems {\bf 36} (2020), 064001

\bibitem{Antil-Diaz-Jing-Schikorra}
{\sc H.\ Antil, H.\ D{\'i}az, T.\ Jing, A.\ Schikorra}, 
{\em Nonlocal bounded variations with applications}, 
SIAM J.\ Math.\ Anal.\ {\bf 56} (2024), 1903--1935

\bibitem{Antiletal2}
{\sc H.\ Antil, R.\ Khatri, R.\ L\"ohner, D.\ Verma}, 
{\em Fractional deep neural network via constrained optimization}, 
Mach.\ Learn.: Sci.\ Technol.\ {\bf 2} (2020), 015003

\bibitem{Bellido-Cueto-MoraCorral4}
{\sc J.C.\ Bellido, J.\ Cueto, C.\ Mora-Corral},
{\em Fractional Piola identity and polyconvexity in fractional spaces}, 
Ann.\ Inst.\ H.\ Poincar{\'e} C Anal.\ Non Lin{\'e}aire {\bf 37} (2020), 955--981

\bibitem{Bellido-Cueto-MoraCorral}
{\sc J.C.\ Bellido, J.\ Cueto, C.\ Mora-Corral}, 
{\em {$\Gamma$}-convergence of polyconvex functionals involving {$s$}-fractional gradients to their local counterparts}, 
Calc.\ Var.\ Partial Differential Equations {\bf 60} (2021), Paper No.\ 7

\bibitem{Bellido-Cueto-MoraCorral3}
{\sc J.C.\ Bellido, J.\ Cueto, C.\ Mora-Corral}, 
{\em Non-local gradients in bounded domains motivated by continuum mechanics: Fundamental theorem of calculus and embeddings}, 
Adv.\ Nonlinear Anal.\ {\bf 12} (2023), Paper No.\ 20220316

\bibitem{Bellido-Cueto-MoraCorral2}
{\sc J.C.\ Bellido, J.\ Cueto, C.\ Mora-Corral}, 
{\em Minimizers of nonlocal polyconvex energies in nonlocal hyperelasticity}, 
Adv.\ Calc.\ Var.\ {\bf 17} (2024), 1039--1055

\bibitem{Bellido-MoraCorral-Pedregal}
{\sc J.C.\ Bellido, C.\ Mora-Corral, P.\ Pedregal}, 
{\em Hyperelasticity as a {$\Gamma$}-limit of peridynamics when the horizon goes to zero}, 
Calc.\ Var.\ Partial Differential Equations {\bf 54} (2015), 1643--1670

\bibitem{Bessas-Stefani}
{\sc K.\ Bessas, G.\ Stefani}, 
{\em Non-local BV functions and a denoising model with {$L^{1}$} fidelity},
Adv.\ Calc.\ Var.\ {\bf 18} (2025), 189--217

\bibitem{BBD24}
{\sc A.\ Braides, G.C.\ Brusca, D.\ Donati},
{\it Another look at elliptic homogenization},
Milan J.\ Math.\ {\bf 92} (2024), 1--23

\bibitem{BDV96}
{\sc A.\ Braides, A.\ Defranceschi, E.\ Vitali},
{\it Homogenization of free discontinuity problems},
Arch.\ Ration.\ Mech.\ Anal.\ {\bf 135} (1996), 297--356

\bibitem{BCCS22}
{\sc E.\ Bruè, M.\ Calzi, G.E.\ Comi, G.\ Stefani},
{\em A distributional approach to fractional Sobolev spaces and fractional variation: asymptotics II},
C.\ R.\ Math.\ Acad.\ Sci.\ Paris {\bf 360} (2022), 589--626

\bibitem{Bungert-Stinson}
{\sc L.\ Bungert, K.\ Stinson}, 
{\em Gamma-convergence of a nonlocal perimeter arising in adversarial machine learning}, 
Calc.\ Var.\ Partial Differential Equations {\bf 63} (2024), Paper No.\ 114

\bibitem{Caffarelli}
{\sc L.\ Caffarelli, J.-M.\ Roquejoffre, O.\ Savin}, 
{\em Nonlocal minimal surfaces}, 
Comm.\ Pure Appl.\ Math.\ {\bf 63} (2010), 1111--1144

\bibitem{CDMSZ}
{\sc F.\ Cagnetti, G.\ Dal Maso, L.\ Scardia, C.I.\ Zeppieri}, 
{\em $\Gamma$-convergence of free-discontinuity problems},
Ann.\ Inst.\ H.\ Poincaré C Anal.\ Non Linéaire {\bf 36} (2019), 1035--1079

\bibitem{Cagnettietal}
{\sc F.\ Cagnetti, G.\ Dal Maso, L.\ Scardia, C.I.\ Zeppieri}, 
{\em A global method for deterministic and stochastic homogenisation in $BV$}, 
Ann.\ PDE {\bf 8} (2022), Paper No.\ 8

\bibitem{Carionietal-2025}
{\sc M.\ Carioni, L.\ Del Grande, J.A.\ Iglesias, H.\ Sch\"onberger}, {\em Nonlocal perimeters and variations: Extremality and decomposability for finite and infinite horizons}, preprint 2025, arXiv: \href{https://arxiv.org/abs/2502.05149}{2502.05149}

\bibitem{Cesaroni-DeLuca-Novaga-Ponsiglione}
{\sc A.\ Cesaroni, L.\ De Luca, M.\ Novaga, M.\ Ponsiglione}, 
{\em Stability results for nonlocal geometric evolutions and limit cases for fractional mean curvature flows}, 
Comm.\ Partial Differential Equations {\bf 46} (2021), 1344--1371

\bibitem{Cesaronietal}
{\sc A.\ Cesaroni, S.\ Dipierro, M.\ Novaga, E.\ Valdinoci}, 
{\em Minimizers for nonlocal perimeters of Minkowski type}, 
Calc.\ Var.\ Partial Differential Equations {\bf 57} (2018), Paper No.\ 64

\bibitem{Cesaronietal2}
{\sc A.\ Cesaroni, S.\ Dipierro, M.\ Novaga, E.\ Valdinoci}, 
{\em Fattening and nonfattening phenomena for planar nonlocal curvature flows}, 
Math.\ Ann.\ {\bf 375} (2019), 687--736

\bibitem{Cesaroni-Novaga-isoperimeter}
{\sc A.\ Cesaroni, M.\ Novaga}, 
{\em The isoperimetric problem for nonlocal perimeters}, 
Discrete Contin.\ Dyn.\ Syst.\ Ser.\ S {\bf 11} (2018), 425--440

\bibitem{Cesaroni-Novaga-clusters}
{\sc A.\ Cesaroni, M.\ Novaga}, 
{\em Nonlocal minimal clusters in the plane}, 
Nonlinear Anal.\ {\bf 199} (2020), 111945

\bibitem{Chambolle-Novaga-Ruffini}
{\sc A.\ Chambolle, M.\ Novaga, B.\ Ruffini}, 
{\em Some results on anisotropic fractional mean curvature flows}, 
Interfaces Free Bound.\ {\bf 19} (2017), 393--415

\bibitem{CS19}
{\sc G.E.\ Comi, G.\ Stefani},
{\em A distributional approach to fractional Sobolev spaces and fractional variation: existence of blow-up}, 
J.\ Funct.\ Anal.\ {\bf 277} (2019), 3373--3435

\bibitem{CS23-1}
{\sc G.E.\ Comi, G.\ Stefani},
{\em A distributional approach to fractional Sobolev spaces and fractional variation: asymptotics I}, 
Rev.\ Mat.\ Complut.\ {\bf 36} (2023), 491--569

\bibitem{CS24}
{\sc G.E.\ Comi, G.\ Stefani},
{\em On sets with finite distributional fractional perimeter}
Anisotropic isoperimetric problems and related topics, 127--150. 
Springer INdAM Ser., 62 Springer, Singapore, (2024)

\bibitem{Cueto-Kreisbeck-Schonberger}
{\sc J.\ Cueto, C.\ Kreisbeck, H.\ Sch{\"o}nberger}, 
{\em A variational theory for integral functionals involving finite-horizon fractional gradients}, 
Fract.\ Calc.\ Appl.\ Anal.\ {\bf 26} (2023), 2001--2056

\bibitem{Cueto-Kreisbeck-Schonberger2}
{\sc J.\ Cueto, C.\ Kreisbeck, H.\ Sch{\"o}nberger}, 
{\em {$\Gamma $}-convergence involving nonlocal gradients with varying horizon: recovery of local and fractional models}, 
Nonlinear Anal.\ Real World Appl.\ {\bf 85} (2025), Paper No.\ 104371

\bibitem{DM93}
{\sc G.\ Dal Maso},
{\em An introduction to $\Gamma$-convergence}, 
Progr.\ Nonlinear Differential Equations Appl., 8. Birkhäuser Boston, Inc., Boston, MA, 1993

\bibitem{DiNezza}
{\sc E.\ Di Nezza, G.\ Palatucci, E.\ Valdinoci}, 
{\em Hitchhiker’s guide to the fractional Sobolev spaces}, 
Bull.\ Sci.\ Math.\ {\bf 136} (2012), 521--573

\bibitem{EKM}
{\sc D.\ Engl, C.\ Kreisbeck, M.\ Morandotti},
{\em Characterizing $BV$-and $BD$-ellipticity for a class of positively 1-homogeneous surface energy densities},
preprint 2024, arXiv: \href{https://arxiv.org/abs/2402.15450}{2402.15450}

\bibitem{Fanizza}
{\sc A.\ Fanizza}, {\em Gamma-convergence as {$s \to 1^{-}$} of anisotropic nonlocal fractional perimeter functionals}, preprint 2025, arXiv: \href{https://arxiv.org/abs/2509.13823}{2509.13823}

\bibitem{Figalli-Fusco-Maggi-Millot-Morini}
{\sc A.\ Figalli, N.\ Fusco, F.\ Maggi, V.\ Millot, M.\ Morini}, 
{\em Isoperimetry and stability properties of balls with respect to nonlocal energies}, 
Comm.\ Math.\ Phys.\ {\bf 336} (2015), 441--507

\bibitem{Figalli-Valdinoci}
{\sc A.\ Figalli, E.\ Valdinoci}, 
{\em Regularity and Bernstein-type results for nonlocal minimal surfaces}, 
J.\ Reine Angew.\ Math.\ {\bf 729} (2017), 263--273

\bibitem{Frank-Seiringer}
{\sc R.L.\ Frank, R.\ Seiringer}, 
{\em Non-linear ground state representations and sharp hardy inequalities}, 
J.\ Funct.\ Anal.\ {\bf 255} (2008), 3407--3430

\bibitem{Friedrich-Seitz-Stefanelli}
{\sc M.\ Friedrich, M.\ Seitz, U.\ Stefanelli}, 
{\em Nonlocal-to-local limit in linearized viscoelasticity}, 
Commun.\ Appl.\ Ind.\ Math.\ {\bf 15} (2004), 1--26

\bibitem{FS20}
{\sc M.\ Friedrich, F.\ Solombrino},
{\em Functionals defined on piecewise rigid functions: integral representation and $\Gamma$-convergence}.
Arch.\ Ration.\ Mech.\ Anal.\ {\bf 236} (2020), 1325--1387

\bibitem{Gilboa-Osher}
{\sc G.\ Gilboa, S.\ Osher}, 
{\em Nonlocal operators with applications to image processing},
Multiscale Model.\ Simul.\ {\bf 7} (2008), 1005--1028

\bibitem{Goffman-Serrin}
{\sc C.\ Goffman, J.\ Serrin},
{\em Sublinear functions of measures and variational integrals.}
Duke Math.\ J.\ {\bf 31} (1964), 159--178

\bibitem{Grasmair10}
{\sc M.\ Grasmair}, 
{\em A coarea formula for anisotropic total variation regularisation},
Industrial Geometry Report no.\ 103 (2010)

\bibitem{Holler-Kunisch}
{\sc G.\ Holler, K.\ Kunisch}, 
{\em Learning nonlocal regularization operators}, 
Math.\ Control Relat.\ Fields {\bf 12} (2022), 81--114

\bibitem{Horvath}
{\sc J.\ Horv{\'at}h}, 
{\em On some composition formulas}, 
Proc.\ Amer.\ Math.\ Soc.\ {\bf 10} (1959), 433--437

\bibitem{Iglesias-Mercier}
{\sc J.A.\ Iglesias, G.\ Mercier}, 
{\em Convergence of level sets in fractional laplacian regularization}
Inverse Problems {\bf 38} (2022), Paper No.\ 124003

\bibitem{KS22}
{\sc C.\ Kreisbeck, H.\ Schönberger},
{\em Quasiconvexity in the fractional calculus of variations: characterization of lower semicontinuity and relaxation},
Nonlinear Anal.\ {\bf 215} (2022), Paper No.\ 112625

\bibitem{Ludwig}
{\sc M.~Ludwig}, {\em Anisotropic fractional perimeters}, J.~Differential Geom., {\bf 96} (2014), 77--93

\bibitem{Mengesha-Du}
{\sc T.\ Mengesha, Q.\ Du}, 
{\em On the variational limit of a class of nonlocal functionals related to peridynamics}, 
Nonlinearity {\bf 28} (2015), 3999--4035

\bibitem{Mizuta96}
{\sc Y.\ Mizuta},
{\em Potential theory in Euclidean spaces},
GAKUTO Internat. Ser. Math. Sci. Appl.\ 6. Gakktosho Co., Ltd., Tokyo, 1996

\bibitem{Schonberger}
{\sc H.\ Sch{\"o}nberger}, 
{\em Extending linear growth functionals to functions of bounded fractional variation}, 
Proc.\ Roy.\ Soc.\ Edinburgh Sect.\ A {\bf 154} (2024), 304--327

\bibitem{ShSp15}
{\sc T.-T.\ Shieh, D.E.\ Spector},
{\em On a new class of fractional partial differential equations},
Adv.\ Calc.\ Var.\ {\bf 8} (2015), 321--336

\bibitem{ShSp18}
{\sc T.-T.\ Shieh, D.E.\ Spector},
{\em On a new class of fractional partial differential equations II},
Adv.\ Calc.\ Var.\ {\bf 11} (2018), 289--307

\bibitem{Silhavy20}
{\sc M.\ Šilhavý},
{\em Fractional vector analysis based on invariance requirements (critique of coordinate approaches)},
Contin.\ Mech.\ Thermodyn.\ {\bf 32} (2020), 207--228

\bibitem{Silling1}
{\sc S.A.\ Silling}, 
{\em Reformulation of elasticity theory for discontinuities and long-range forces}, J.\ Mech.\ Phys.\ Solids {\bf 48} (2000), 175--209

\bibitem{Silling2}
{\sc S.A.\ Silling}, {\em Linearized theory of peridynamic states}, 
J.\ Elasticity {\bf 99} (2010), 85--111

\bibitem{Silling3}
{\sc S.A.\ Silling, R.B.\ Lehoucq}, 
{\em Peridynamic theory of solid mechanics}, 
Advances in Applied Mechanics {\bf 44} (2010), 73--168

\end{thebibliography}
\end{document}